\documentclass{article}

\usepackage{amsmath}
\usepackage{amsthm}
\usepackage{amsfonts}
\usepackage{amssymb}
\usepackage{graphicx}
\usepackage{latexsym}
\usepackage{mathrsfs}

\newtheorem{theorem}{Theorem}[section]
\newtheorem{thm}{Theorem}[section]

\newtheorem{lem}[theorem]{Lemma}

\newtheorem{prop}[theorem]{Proposition}

\theoremstyle{definition}

\newtheorem{defn}[theorem]{Definition}

\newtheorem{cor}[theorem]{Corollary}

\theoremstyle{remark}
\newtheorem{remark}[theorem]{Remark}

\numberwithin{equation}{section}

\newcommand\bA{\mathbb A}

\newcommand\MM{\mathbb M}
\newcommand\NN{\mathbb N}

\newcommand\RR{\mathbb R}
\newcommand\ZZ{\mathbb Z}

\newcommand\EE{\mathbb E}
\newcommand\FF{\mathbb F}
\newcommand\F{\mathbb F}

\newcommand\cB{\mathcal{B}}

\newcommand\cG{\mathcal{G}}

\newcommand\cI{\mathcal{I}}

\newcommand\cS{\mathcal{S}}

\newcommand\cF{\mathcal{F}}

\newcommand\cR{\mathcal{R}}

\newcommand\Inn{\operatorname{Inn}}

\newcommand\Par{\operatorname{P}}

\newcommand\tphi{\tilde\phi}

\newcommand\tY{\widetilde Y}
\newcommand\tZ{\widetilde Z}

\newcommand\norm[1]{\left\|#1\right\|}

\newcommand\abs[1]{\left|#1\right|}

\def\cc{{\curvearrowright}}

\newcommand\sS{\textnormal{S}}

\def\cc{\curvearrowright}
\def\f{{\textrm{fin}}}

\def\tcB{{\tilde {\mathcal B}}}
\def\tcS{{\tilde {\mathcal S}}}
\def\tPhi{\tilde \Phi}

\begin{document}
\title{Geometric covering arguments and ergodic theorems for free groups}
\author{Lewis Bowen\footnote{supported in part by NSF grant DMS-0968762, NSF CAREER Award DMS-0954606 and BSF grant 2008274} ~and Amos Nevo\footnote{supported in part by ISF 776-09 grant  and BSF grant 2008274}}






\maketitle
\begin{abstract}
We present a new approach to the proof of ergodic theorems for actions of free groups which generalize
 the classical geometric covering and asymptotic invariance arguments used in the ergodic theory of amenable groups. 
Existing maximal and pointwise ergodic theorems for free group actions are extended to a large class of geometric averages which were not accessible by previous techniques. 
 \end{abstract}

\tableofcontents

\section{Introduction}

Let $G$ be a locally compact second countable group with Haar measure $m$, and let $B_t$ for $t\in \NN$ or $t\in \RR$ be a family Borel sets of positive finite measure. Let  $\mu_t$ be probability measures supported on $B_t$.  Suppose $G$ acts by measure-preserving transformations on a probability space $(X,\lambda)$. For any $f \in L^1(X,\lambda)$ we may consider the averaging operator 
$$\bA_t[f](x):= \int_{B_t} f(g^{-1}x)~d\mu_t(g).$$
Let $\EE[f|G]$ denote the conditional expectation of $f$ with respect to the $\sigma$-algebra of $G$-invariant subsets. We say that $\{\mu_t\}$ is a {\em pointwise ergodic family in $L^p$} if 
$\bA_t[f]$ converges to $\EE[f|G]$ pointwise almost everywhere and in $L^p$-norm  for every $f\in L^p(X,\lambda)$ and for every measure-preserving action of $G$ on a probability space $(X,\lambda)$. 

The most useful pointwise ergodic families are those in which $B_t$ are naturally connected with the geometry of the group.  A basic case to consider is when $B_t$ is the ball of radius $t> 0$ with respect to an invariant metric, and $\mu_t$ is the Haar-uniform probability measure on $B_t$, namely the density of $\mu_t$ is $\chi_{B_t}/m(B_t)$. Such  averages are referred to as ball averages. Spherical and shell averages are defined similarly.

Most of the research on ergodic theorems  has focused on the case when the group  is amenable and the averages are Haar-uniformly distributed on sets which form an asymptotically invariant  (F\o lner) sequence.  
The covering properties of translates of these sets and their property of asymptotic invariance play an indispensable role in the arguments developed in the amenable case, and we refer to \cite{N3} for a detailed survey of these methods and current results. 

In contrast, non-amenable groups do not admit asymptotically invariant sequences, 
 and so the arguments developed to handle amenable groups are not directly applicable. An alternative  general approach to the ergodic theory of group actions based on the spectral theory of unitary representations was developed and applied to the case where $G$ is a semisimple $S$-algebraic group, or a lattice subgroup of such a group. We refer to \cite{GN1} for a detailed account of this theory and to \cite{GN2} for some of its applications. Naturally, reliance on harmonic analysis techniques limits the scope of this theory to groups whose unitary representation theory can be explicated, and to their lattice subgroups. 
 
 For general groups, and certainly for discrete groups such as (non-elementary) word-hyperbolic groups for example, spectral information is usually unavailable and harmonic analysis techniques are usually inapplicable. Exceptions do exist, and for example it was proven by spectral methods that ball averages with respect to certain invariant metrics on the free group do indeed form pointwise ergodic sequences. 
 The metrics allowed are those arising from first fixing an embedding of the free group as a lattice in a locally compact group. Thus in \cite{N1}\cite{NS} the free group is viewed as a lattice in the group of automorphisms of a regular tree, and  in \cite{GN1} as a lattice in $PSL_2(\RR)$, and the metric is obtained by restricting a suitable $G$-invariant metric to the lattice subgroup. 
Note that in the case of the tree metric a periodicity phenomenon arise, namely the balls form a pointwise ergodic sequence if and only if the sign character of the free group does not appear in the spectrum. 
 
 In the case of the tree metric, a proof of the ergodic theorem in $L\log L$ was given by \cite{Bu}, using Markov operators (inspired by earlier related ideas in \cite{Gr99}). This method extends to groups with a Markov presentation (which include all hyperbolic groups). The averaging sequences obtained are related to the Markov presentation rather than a metric structure on the group.

  In the present paper we develop a new approach to pointwise ergodic theorems for actions of free groups,  based on geometric covering and asymptotic invariance arguments. This approach has two significant advantages: first, it constitutes a direct generalization of the classical arguments employed to prove ergodic theorems for amenable groups, and in fact reduces the proof of ergodic theorems for the free group to the proof of ergodic theorems for a certain amenable equivalence relation. Second, as will be shown in forthcoming work, the new ideas extend beyond the class of free groups to word-hyperbolic groups \cite{BN1}, semisimple Lie groups \cite{BN2}, and others. 
  
  
   Our goal in what follows is to explain our method in detail in the most accessible case, namely that of free groups, and show how to use it to generalize the existing ergodic theorems on free groups. The main results establish maximal inequalities and pointwise convergence for a wide class of geometrically defined averages not accessible by previous techniques, one simple example being sector averages (defined below). We also establish the integrability of the maximal function associated with these sequences when the original function is in $L\log L$, and thus also pointwise convergence of the averages acting on functions in this space.

\subsection{Statement of the main theorems}\label{sec:state}
Let $\FF=\langle a_1,\dots,a_r \rangle$ denote the free group on $r$ generators. Let $\sS=\{a_i, a_i^{-1}\}_{i=1}^r $ be the associated symmetric generating set. For every nonidentity element $g\in \FF$, there is a unique sequence $t_1,\ldots, t_n$ of elements in $\sS$ such that $g=t_1\cdots t_n$ and $n \ge 1$ is as small as possible. Define $|g|=n$. Let $\partial \FF$ be the boundary of $\FF$ which we identify with the set of all infinite sequences $(s_1,s_2,\ldots) \in \sS^\NN$ such that $s_{i+1} \ne s_i^{-1}$ for all $i \ge 1$. If $g=t_1\cdots t_n$ as above then the {\em shadow of $g$} (with light source at $e$)  is the compact open set 
$$O_g=\big\{ (s_1,s_2,\ldots) \in \partial \FF :~ s_i=t_i ~\text{ for }1\le i \le n\big\}.$$
The boundary admits a natural probability measure $\nu$ such that $\nu(O_g)=(2r)^{-1}(2r-1)^{-|g|+1}.$  

We denote the sphere of radius $n$ in $\FF$ by  $S_n(e)=\{g\in \FF:~|g|=n\}$. 
Let $\psi$ be any probability density function on $\partial \FF$; namely $\psi \ge 0$ and $\int_{\partial \FF}\psi~d\nu=1$. Define the associated probability measures $\mu_n^\psi$ on $S_n(e)$ given by  $\mu_n^\psi(g)=\int_{O_g}\psi ~d\nu$.

Let $\FF^2 < \FF$ be the subgroup generated by all elements $g$ such that $|g|$ is even. It is a subgroup of index $2$ in $\FF$. Given a probability space $(X,\lambda)$ on which $\FF$ acts by measure-preserving transformations, we let $\EE[f|\FF^2]$ denote the conditional expectation of a function $f \in L^1(X,\lambda)$ on the $\sigma$-algebra of $\FF^2$-invariant sets.

\begin{thm}\label{thm:sector}
Fix any continuous probability density function $\psi$ on the boundary $\partial \FF$. Then in any measure-preserving action of $\FF$ on a standard probability space $(X,\lambda)$,  and any $f\in L^p(X)$ for $1<p<\infty$, the averages $\mu_{2n}^\psi(f) \in L^p(X)$ defined by
$$\mu_{2n}^\psi(f)(x):=\sum_{g \in S_{2n}(e)} f(g^{-1}x)\mu^\psi_{2n}(g)$$
 converge pointwise almost surely and in $L^p$-norm  to $\EE[f|\FF^2]$. Furthermore, pointwise convergence to the same limit holds for any $f$ in the Orlicz space $(L\log L)(X,\lambda)$.
\end{thm}
\begin{remark}
In the special case in which the density is identically $1$,  each $\mu_{2n}$ is the uniform average on $S_{2n}(e)$, and the theorem states that even-radius spherical averages converge pointwise a.e. to $\EE[f|\FF^2]$, for all $f\in L^p$, $1 < p < \infty$ and $f\in L\log L$. The proof of Theorem \ref{thm:sector} is completely different and independent of the  previous proofs of this fact in \cite{N1}, \cite{NS} and \cite{Bu}. 
\end{remark}
\begin{remark}
Given $w \in \FF$, we can choose the density $\rho_w=\chi_{O_w}/\nu(O_w)$ to be the normalized characteristic function of a basic compact open subset $O_w$ of $\partial \FF$. Thus, the sequence 
$\mu^{\rho_w}_{2n}$ of uniform averages  on 
 the set of all words of length $2n\ge \abs{w}$ with initial subword $w$ is a pointwise ergodic sequence. It is natural to call these averages (in analogy with the hyperbolic plane) {\em sector averages}. 
\end{remark}
Theorem \ref{thm:sector} is a special case of a more general result, whose statement requires further notation.  For $g\in \FF$, let $\delta_g \in \ell^1(\F)$ be the function $\delta_g(g')=1$ if $g=g'$ and $0$ otherwise. Let $\pi_\partial :\ell^1(\F) \to L^1(\partial \FF,\nu)$ be the linear map satisfying $\pi_\partial(\delta_g)= \nu(O_g)^{-1}\chi_{O_g}$ where $\chi_{O_g}$ is the characteristic function of $O_g$. If  $\mu \in \ell^1(\F)$ and $\mu \ge 0$ then $\pi_\partial(\mu) \ge 0$ and $\|\pi_\partial(\mu)\|_1=\|\mu\|_1$.

\begin{thm}\label{thm:main}
Let $\{\mu_{2n}\}_{n=1}^\infty$ be a sequence of probability measures in $\ell^1(\FF)$ such that $\mu_{2n}$ is supported on the sphere $S_{2n}(e)$. 
Let $1< q < \infty$,  and suppose $\{\pi_\partial(\mu_{2n})\}_{n=1}^\infty$ converges in $L^q(\partial \FF,\nu)$. Let $(X,\lambda)$ be a probability space on which $\FF$ acts by measure-preserving transformations.
If $f\in L^p(X)$, $1<p <\infty$ and $\frac{1}{p} + \frac{1}{q} < 1$, then the sequence $\{\mu_{2n}(f)\}_{n=1}^\infty \subset L^p(X)$ defined by 
$$\mu_{2n}(f)(x) := \sum_{g \in S_{2n}(e)} f(g^{-1}x)\mu_{2n}(g)$$
 converges pointwise almost surely and in $L^p$-norm to $\EE[f|\FF^2]$. Furthermore, if $q=\infty$ and  $\{\pi_\partial(\mu_{2n})\}_{n=1}^\infty$ converges uniformly, then pointwise convergence to the same limit holds for any $f$ in the Orlicz space $(L\log L)(X,\lambda)$. 
\end{thm}

Theorem \ref{thm:sector} follows from Theorem \ref{thm:main}. To see this, fix a continuous probability density  $\psi$ on $\partial \FF$. Then the associated averages $\mu_{n}^\psi$ (defined above) satisfy 
$\lim_{n\to \infty} \pi_\partial(\mu_n^\psi)=\psi$ in $L^q(\partial \FF, \nu)$, for all $1 \le q  \le \infty$. Indeed,  the continuous functions $\pi_\partial( \mu_n^\psi )$ converge to $\psi$ uniformly. Hence Theorem \ref{thm:main} applies. 


There are several natural questions raised by Theorem \ref{thm:main}. For example,  does the conclusion of Theorem \ref{thm:main} hold if the hypothesis that $\mu_{2n}$ is supported on $\{g \in \FF:~|g|=2n\}$ is weakened to the condition $\lim_{n\to\infty} \mu_{2n}(g)=0$ for all $g\in \FF$? Does it hold if the inequality $\frac{1}{p} + \frac{1}{q} <1$ is replaced by the weaker constraint $\frac{1}{p} + \frac{1}{q} \le 1$? What if instead of being convergent in $L^q(\partial \FF,\nu)$, $\{\pi_\partial(\mu_{2n})\}_{n=1}^\infty$ is only required to be pre-compact or norm-bounded?



\subsection{On the ideas behind the proof}

To illustrate our approach, consider the following scenario. Suppose that $G$ is a group and $H<G$ is a subgroup. We say that $H$ has the {\em automatic ergodicity property} if whenever $G$ acts on a probability space $(X,\mu)$ by measure-preserving transformations ergodically then the action restricted to $H$ is also ergodic. In this case, any pointwise ergodic sequence for $H$ is a pointwise ergodic sequence for $G$. If $H$ is amenable then we can use the classical theory to find such sequences in $H$. Then conjugate copies of pointwise ergodic sequences can be averaged to construct additional pointwise ergodic sequences supported on $G$. 

For example, if $G=SL_2(\RR)$ then, by the Howe-Moore Theorem, any closed noncompact subgroup $H<G$ has the automatic ergodicity property. In \cite{BN2} we use the foregoing observation to prove pointwise ergodic theorems by averaging on conjugates of a horospherical (unipotent) subgroup, which is isomorphic to $\RR$. Similar considerations apply to other Lie groups as well. 

To handle free groups we will have to modify this approach by considering an appropriately chosen amenable ``measurable subgroup''. This ``subgroup'' is a probability measure on the space of horospheres containing the identity element. Such horospheres have an intrinsic geometric structure and a natural notion of asymptotically invariant (F\o lner) sequence. We develop variants of the classical covering arguments and establish pointwise convergence for averaging on  F\o lner sequences along  horospheres. The space of all horospheres containing the identity is identifiable with the boundary $\partial \FF$. We show that this ``measurable subgroup'' satisfies an analogue of automatic ergodicity (a more general result is proven in [Bo08] for all word hyperbolic groups). This is related to the fact that the action of $\FF$ on its boundary is weakly mixing [AL05]. By averaging appropriately chosen horospherical F\o lner sequences over the space of horospheres, we obtain that the uniform measures on spheres form a pointwise ergodic sequence for $\FF$. Our approach allows much more general types of averaging sequences to be analyzed similarly, since we can average the horospherical sequences with respect to a variety of measures on the boundary.

\subsection{Outline of the paper}
We begin by proving ergodic theorems for equivalence relations in \S \ref{sec:equivalence}. This involves a direct generalization of classical arguments. In \S \ref{sec:free group review} we review the boundary of $\FF$, horospheres and horofunctions on $\FF$. After these preliminaries, we state a pointwise ergodic theorem for averages along horospheres (\S  \ref{sec:horo1}) and prove it using results of \S \ref{sec:equivalence}. In \S \ref{sec:ergodicity} we turn to ergodicity and periodicity, and prove that an ergodic action of $\FF$ gives rise to a `virtually ergodic' action of the associated measurable subgroup.  In the last section we analyze integration of  the horospherical averages over the space of horospheres and prove Theorem \ref{thm:main}.




\section{An ergodic theorem for equivalence relations}\label{sec:equivalence}
Let $(B,\nu)$ be a standard Borel probability space and $\cR(B) \subset B\times B$  be a Borel equivalence relation (i.e., for all $b,b',b'' \in B$, $(b,b) \in \cR(B)$, $(b,b') \in \cR(B) \Rightarrow (b',b) \in \cR_B$, $(b,b'), (b',b'') \in \cR(B) \Rightarrow (b,b'') \in \cR(B)$), with countable equivalence classes. Let $c$ denote counting measure on $B$ (so $c(E) = \# E ~\forall E \subset B$). The measure $\nu$ on $B$ is {\em $\cR(B)$-invariant} if $\nu\times c$ restricted to $\cR(B)$ equals $c\times \nu$ restricted to $\cR(B)$. A Borel map $\phi: B \to B$ is an {\em inner automorphism} of $\cR(B)$ if it is invertible with Borel inverse and its graph is contained in $\cR(B)$. Let $\Inn(\cR(B))$ denote the group of inner automorphisms. If $\nu$ is $\cR(B)$-invariant then $\phi_*\nu=\nu$ for every $\phi \in \Inn(\cR(B))$. For the rest of this section, we assume $\nu$ is a $\cR(B)$-invariant Borel probability measure on $B$.

A basic example to keep in mind is the following special case: suppose $G$ is a discrete group acting my measure-preserving transformations on $(B,\nu)$. Then the orbit-equivalence relation $\cR(B):=\{(b,gb):~b\in B, g\in G\}$ is such that $\nu$ is $\cR(B)$-invariant. In fact, a result of \cite{FM77} implies that all probability measure-preserving discrete equivalence relations arise from this construction (up to isomorphism).

Suppose that $\cF=\{\cF_n\}_{n=1}^\infty$ is a sequence of functions $\cF_n:B \to 2_\f^B$ (where $2_\f^B$ denotes the space of finite subsets of $B$) such that for each $n$, $\{(b,b'):~b' \in \cF_n(b)\}\subset B \times B$ is a Borel subset of $\cR(B)$. We are concerned with several properties such a sequence could satisfy:
 
\begin{enumerate}
\item A set $\Phi \subset \Inn(\cR(B))$ {\em generates $\cR(B)$} with respect to $\nu$ if for $\nu\times c$ a.e. $(b_1,b_2) \in \cR(B)$ there exists $\phi \in \langle \Phi \rangle$ such that $\phi(b_1)=b_2$ (where $\langle \Phi \rangle$ denotes the group generated by $\Phi$).
\item $\cF$ is {\em asymptotically invariant} (or {\em F\o lner}) with respect to $\nu$ if there exists a countable set $\Phi \subset \Inn(\cR(B))$ which generates $\cR(B)$ such that
$$\lim_{n\to\infty} \frac{|\cF_n(b) \Delta \phi(\cF_n(b))|}{|\cF_n(b)|} =0\quad \forall \phi \in \Phi, \, \textrm{$\nu$-a.e. }  b \in B.$$

\item $\cF$ is {\em non-shrinking} with respect to $\nu$ if there is a constant $C_s>0$ such that for any Borel  $Y \subset B$ and any bounded measurable function $\rho:Y \to \NN$ we have
$$\nu\Big( \bigcup\{ \cF_{\rho(y)}(y):~y\in Y\}\Big) \ge C_s \nu(Y).$$
This property is trivially satisfied if $b \in \cF_n(b)$ for all $n,b$, which is often the case in practice.

\item $\cF$ satisfies the {\em doubling condition} with respect to $\nu$ if there is a constant $C_d>0$ such that for $\nu$-a.e. $b\in B$ and every $n\in \NN$ 
$$ \Big|\bigcup \big\{\cF_m(b'):~ m \le n,  \cF_m(b') \cap \cF_n(b) \ne \emptyset \big\}\Big| \le C_d|\cF_n(b)|.$$
 \end{enumerate}
For a function $f$ on $B$, consider the averages $\bA_n[\cF;f]$ defined by
$$\bA_n[\cF;f](b):=\frac{1}{|\cF_n(b)|} \sum_{b'\in\cF_n(b)} f(b').$$
We are interested in the convergence properties of these averages. To explain what the limit function could be we need a few definitions: a set $E \subset B$ is {\em $\cR(B)$-invariant} if $E \times B \cap \cR(B) = E \times E$. For a Borel function $f$ on $B$, let $\EE[f|\cR(B)]$ denote the conditional expectation of $f$ with respect to the $\sigma$-algebra of $\cR(B)$-invariant Borel sets and the measure $\nu$.

The purpose of this section is to prove: 

\begin{thm}\label{thm:pointwiseeq}
If $\cF$ is asymptotically invariant, non-shrinking and satisfies the doubling condition then $\cF$ is a pointwise ergodic sequence in $L^1$. I.e., for every $f\in L^1(B,\nu)$, $\bA_n[\cF;f]$ converges pointwise a.e. and in $L^1$-norm to $\EE[f|\cR(B)]$ as $n\to\infty$.
\end{thm}
This result generalizes a classical ergodic theorem for amenable groups as follows. Suppose $\cR(B)$ is the orbit-equivalence relation of an amenable group $G$ acting on $(B,\nu)$ by measure-preserving transformations and $\{F_n\}_{n=1}^\infty$ is a sequence of F\o lner subsets of $G$ (so $\lim_{n\to\infty} \frac{|KF_n \Delta F_n|}{|F_n|} = 0$ for any finite $K \subset G$). Let $\cF_n(b):=\{fb:~ f\in F_n\}$. Then $\cF=\{\cF_n\}_{n=1}^\infty$ is asymptotically invariant. The doubling condition for $\{F_n\}_{n=1}^\infty$ can be stated as: $|\cup_{m \le n} F_m^{-1}F_n| \le C_d |F_n|$. This implies that $\cF$ is doubling. It is usually assumed in classical theorems that the identity element $e \in F_n$ for all $n$. This implies $\cF$ is non-shrinking. So the theorem above implies an ergodic theorem for amenable groups with respect to a doubling F\o lner sequence of sets containing the identity element. This is not the most general ergodic theorem known for amenable groups (see e.g., \cite{Li01}).

Theorem \ref{thm:pointwiseeq} is obtained from the next two theorems which are also proven in this section.
\begin{thm}[Dense set of good functions]\label{thm:dense}
If $\cF$ is asymptotically invariant then there exists a dense set $\cG \subset L^1(B)$ such that for all $f \in \cG$, $\bA_n[\cF;f]$ converges pointwise a.e. to $\EE[f|\cR_{ B}]$. Moreover, if $L^1_0(B)$ is the set of all functions $f\in L^1(B)$ with $\EE[f|\cR(B)]=0$ a.e. then there exists a dense set $\cG_0 \subset L^1_0(B)$ such that for all $f\in \cG_0$,  $\bA_n[\cF;f]$ converges pointwise a.e. to $0$.
\end{thm}

\begin{thm}[$L^1$ maximal inequality]\label{thm:maximaleq}
Suppose that $\cF$ is non-shrinking and satisfies the doubling condition. For $f\in L^1(B)$, let $\MM[\cF;f]$ be the maximal function
$$\MM[\cF;f]:=\sup_n \bA_n[\cF;|f|].$$
Then there exists a constant $C>0$ such that for any $f \in L^1(B)$ and any $t>0$,
$$\nu\left(\left\{ b\in B:~\MM[\cF;f](b)>t \right\}\right) \le \frac{C||f||_1}{t}.$$
\end{thm}
Assuming the two Theorems above, we  prove Theorem \ref{thm:pointwiseeq}.
\begin{proof}[Proof of Theorem \ref{thm:pointwiseeq}]
Let $f\in L^1(B)$. We will show that $\{\bA_n[\cF;f]\}_{n=1}^\infty$ converges pointwise a.e. $\EE[f|\cR(B)]$. After replacing $f$ with $f-\EE[f|\cR(B)]$ if necessary we may assume that $\EE[f|\cR(B)]=0$ a.e.

For $t>0$, let $E_t :=\{b \in B:~ \limsup_{n\to\infty} |\bA_n[\cF;f](b)| \le t\}$. We will show that each $E_t$ has measure one. Let $\epsilon=\frac{t^2}{4}$. According to Theorem \ref{thm:dense}, there exists a function $f_1\in L^1(B)$ with $\|f-f_1\|_1 < \epsilon$ such that $\{\bA_n[\cF;f_1]\}_{n=1}^\infty$ converges pointwise a.e. to $0$. Let $n>0$. Observe that
\begin{eqnarray*}
|\bA_n[\cF;f]| &\le& |\bA_n[\cF;f-f_1]| + |\bA_n[\cF;f_1]|\le \MM[\cF;f-f_1] + |\bA_n[\cF;f_1]|.
\end{eqnarray*}
Let 
$$D:=\big\{b \in B:~\MM[\cF;f-f_1](b) \le \sqrt{\epsilon}\big\}.$$
 Since $\bA_n[\cF;f_1]$ converges pointwise a.e. to zero, for a.e. $b\in D$ there is an $N>0$ such that $n>N$ implies
\begin{eqnarray*}
|\bA_n[\cF;f](b)| &\le& \MM[\cF;f-f_1](b) + |\bA_n[\cF;f_1](b)| \le 2\sqrt{\epsilon} =t.
\end{eqnarray*}
Hence $D \subset E_t$ (up to a set of measure zero). By Theorem \ref{thm:maximaleq}, 
$$\nu(E_t) \ge \nu(D) \ge 1-C\epsilon^{-1/2}\|f-f_1\|_1>1-\sqrt{\epsilon} C=1-\frac{Ct}{2}.$$
For any $s<t$, $E_s \subset E_t$. So $\nu(E_t) \ge \nu(E_s) \ge 1-\frac{Cs}{2}$ for all $s<t$ which implies $ \nu(E_t)=1$. So the set $E:=\cap_{n=1}^\infty E_{1/n}$ has full measure. This implies pointwise convergence of $\{\bA_n[\cF;f]\}_{n=1}^\infty$.

The fact that $\bA_n[\cF;f]$ converges to $\EE[f|\cR(B)]$ in $L^1(B)$ follows from the pointwise result. To see this, observe that it is true if $f\in L^\infty(B)$ by the bounded convergence theorem. Since $L^\infty$ is dense in $L^1$ and $A^\cF_{n}$ is a contraction in $L^1$ this implies the result.
\end{proof}

\subsection{A dense set of good functions}
In this subsection, we prove Theorem \ref{thm:dense} by generalizing von Neumann's classical argument. So assume $\cF$ is asymptotically invariant. Let $\Phi \subset \Inn(\cR(B))$ be a countable set generating $\cR(B)$ witnessing the asymptotic invariance. 
\begin{lem}\label{lem:dense1}
Let $\phi \in \langle \Phi \rangle$, $f\in L^\infty(B)$ and define $f' := f - f\circ \phi$. Then  $\bA_n[\cF;f']$ converges pointwise a.e. to $\EE[f'|\cR(B)]$.
\end{lem}

\begin{proof}
First, suppose $\phi \in \Phi$. Because $\cF$ is asymptotically invariant, for a.e. $b \in B$,
\begin{eqnarray*}
\lim_{n\to\infty}\big|\bA_n[\cF;f'](b)\big| &=& \lim_{n\to\infty} \Big|\frac{1}{|\cF_n(b)|} \sum_{b' \in \cF_n(b)} f(b') - f(\phi(b'))\Big|\\
 &\le& 2||f||_\infty\lim_{n\to\infty} \frac{|\cF_n(b) \Delta \phi(\cF_n(b))|}{|\cF_n(b)|} =0.
 \end{eqnarray*}
Since $\nu$ is $\cR(B)$-invariant, $\EE[f|\cR(B)] = \EE[f\circ\phi|\cR(B)]$. Hence $\EE[f'|\cR(B)]=0$ a.e. This proves the lemma in the case $\phi \in \Phi$.

If $f'' := f - f\circ \phi^{-1}$ then $f''\circ \phi = -f'$. So $\bA_n[\cF;f'' \circ\phi]$ converges as $n\to\infty$ to the constant $0$ pointwise a.e. Since $\bA_n[\cF_n;f'' - f''\circ\phi]$ also converges as $n\to\infty$ to $0$ pointwise a.e. (by the previous paragraph), it follows that $\bA_n[\cF;f'']$ converges pointwise a.e. to $0$ as $n\to\infty$. So the lemma is true for $\phi \in \Phi^{-1}$. 

Now suppose that the lemma is true for two functions $\phi_1,\phi_2 \in \Inn(\cR(B))$. It suffices to show that if $f \in L^\infty(B)$ and $f' := f - f\circ \phi_1\circ \phi_2$ then $A_n^\cF[f']$ converges to $0$ pointwise a.e.. This follows from the decomposition
$$f' =[ f- f\circ \phi_1] + [f\circ \phi_1 - f\circ \phi_1\circ \phi_2]$$
and the hypotheses on $\phi_1,\phi_2$. 
\end{proof}



\begin{lem}
Let $f$ be a measurable function on $X$ such that for every $\phi \in \langle \Phi\rangle$, $f=f\circ \phi$ a.e. Then $f$ is $\cR(B)$-invariant. I.e., $f(b)=f(b')$ for a.e. $(b,b') \in \cR(B)$ (with respect to $\nu \times c$ where $c$ is counting measure).
\end{lem}

\begin{proof}
For each $\phi \in \langle \Phi\rangle $, let 
$$B_\phi=\{b \in B:~ f(b) \ne f\circ\phi(b)\}.$$
Since $\Phi$ is countable, the group $\langle \Phi \rangle $ is also countable and 
$$\nu\Big(\bigcup_{\phi \in \langle \Phi \rangle} B_\phi \Big)=0.$$
By definition if $b \notin \bigcup_{\phi \in \langle \Phi \rangle} B_\phi$, then $f(b) = f(\phi(b))$ for all $\phi \in \langle \Phi \rangle$. But this implies $f(b)=f(b')$ for $\nu \times c$-a.e. $(b,b') \in \cR(B)$. 
\end{proof}

\begin{proof}[Proof of Theorem \ref{thm:dense}]
Let $\cI \subset L^2(B)$ be the space of $\cR(B)$-invariant $L^2$ functions. That is, $f \in \cI$ if and only if $f(b)=b'$ for a.e. $\big( b,b'\big) \in \cR_{ B}$. Let $\cG_0 \subset L^2(B)$ be the space of all functions of the form $f-f\circ \phi$ for $f\in L^\infty(B)$ and $\phi \in \langle \Phi \rangle$. We claim that the span of $\cI$ and $\cG_0$ is dense in $L^2(B)$. To see this, let $f_*$ be a function in the orthocomplement of $\cG_0$. Denoting the $L^2$ inner product by $\langle \cdot,\cdot \rangle$, we have
$$0=\langle f_*, f-f\circ\phi\rangle = \langle f_*,f\rangle - \langle f_*, f\circ \phi\rangle = \langle f_*,f\rangle - \langle f_* \circ \phi^{-1}, f\rangle = \langle f_*-f_*\circ\phi^{-1},f\rangle$$
for any $f \in L^\infty(B)$ and  $\phi \in \langle \Phi \rangle$. Since $L^\infty(B)$ is dense in $L^2(B)$, we have $f_*=f_*\circ \phi^{-1}$ for all $\phi \in \langle \Phi \rangle$. So the previous lemma implies $f_*$ is $\cR(B)$-invariant; i.e., $f_* \in \cI$. This implies $\cI$ and $\cG_0$ span $L^2( B)$ as claimed. 

By Lemma \ref{lem:dense1} for every $f\in \cI + \cG_0$, $\bA_n[\cF;f]$ converges pointwise a.e. to $\EE[f|\cR(B)]$. Since $\cI+\cG_0$ is dense in $L^2( B)$, which is dense in $L^1( B)$, the first statement of the theorem follows. The second is similar.
\end{proof}

\subsection{A maximal inequality}

This subsection proves Theorem \ref{thm:maximaleq}. We begin with a covering lemma generalizing the classical Wiener covering argument.

\begin{lem}\label{lem:covering}
Suppose $\cF$ satisfies the doubling condition with constant $C_d>0$. Let $\rho:Y \to \NN$ be a bounded measurable function where $Y \subset B$. Then there exists a measurable set $Z \subset Y$ such that 
\begin{enumerate}
\item for all $z_1, z_2 \in Z$, if $z_1\neq z_2$ then $\cF_{\rho(z_1)}(z_1) \cap \cF_{\rho(z_2)}(z_2) = \emptyset$;
\item 
$$C_d \nu \Big( \bigcup_{z\in Z} \cF_{\rho(z)}(z) \Big) \ge  \nu \Big( \bigcup_{y\in Y} \cF_{\rho(y)}(y) \Big).$$
\end{enumerate}
\end{lem}

\begin{proof}
Let $T:B \to \RR$ be an injective Borel function. We will use $T$ to break `ties' in what follows. 

If $Y' \subset Y$ is a Borel set then we let $M(Y')$  be the set of all `maximal' elements of $Y'$. Precisely, $y_1 \in M(Y')$ if $y_1 \in Y'$ and for all $y_2\in Y'$ different than $y_1$ either
\begin{enumerate}
\item  $\cF_{\rho(y_1)}(y_1) \cap \cF_{\rho(y_2)}(y_2) =\emptyset$,
\item $\rho(y_1) > \rho(y_2)$ or
\item  $\cF_{\rho(y_1)}(y_1) \cap \cF_{\rho(y_2)}(y_2) \ne \emptyset$, $\rho(y_1) = \rho(y_2)$ and $T(y_1) > T(y_2)$.
\end{enumerate}
Because $\rho$ is bounded, if $Y'$ is non-empty then $M(Y')$ is nonempty. 

Let $Y_0:=Y$ and $M_0:=M(Y_0)$. Assuming that $Y_n, M_n \subset Y$ have been defined, let 
$$Y_{n+1}:=\{y \in Y:~ \cF_{\rho(y)}(y) \cap \cF_{\rho(z)}(z)=\emptyset ~\forall z \in M_n\}$$
and $M_{n+1}=M(Y_{n+1})$.  Let 
$$Z= \bigcup_n M_n, ~\tY := \bigcup_{y\in Y} \cF_{\rho(y)}(y), ~\tZ := \bigcup_{z\in Z} \cF_{\rho(z)}(z).$$
By construction, for all $z_1, z_2 \in Z$ if $z_1\neq z_2$ then $\cF_{\rho(z_1)}(z_1) \cap \cF_{\rho(z_2)}(z_2) = \emptyset$. 

For each $z\in M_n$, let 
$$S(z):= \bigcup\{ \cF_{\rho(y)}(y):~ y\in Y_n, ~\cF_{\rho(y)}(y) \cap \cF_{\rho(z)}(z) \ne \emptyset\}.$$
Note that $Z$ is the disjoint union of all the $M_n$'s. So $S$ is well-defined as a function on $Z$. Since each $z\in M_n$ is maximal in $Y_n$, the doubling condition implies
$$|S(z)| \le C_d |\cF_{\rho(z)}(z)|.$$
The construction of $Z$  implies  $\tY = \bigcup_{z\in Z} S(z)$. So there exists a measurable function $J: \tY \to \tZ$ such that
\begin{itemize}
\item for each $y\in \tY$, $J(y) \in \cF_{\rho(z)}(z)$ where $z \in Z$ is an element with $y \in S(z)$;
\item $|J^{-1}(z)| \le C_d$ for all $z\in \tZ$.
\end{itemize}
For example, for each $z \in Z$, let $J_z: S(z) \to \cF_{\rho(z)}(z)$ be a map so that $|J^{-1}_z(z')| \le C_d$ for all $z' \in \cF_{\rho(z)}(z)$. This family of maps can be chosen to vary in Borel manner with respect to $z$. For any $y\in \tY$ define $J(y):=J_z(y)$ where $z \in Z$ is the unique element satisfying
\begin{itemize}
\item $y \in S(z)$,
\item $T(z) \ge T(z')$ for all $z' \in Z$ with $y \in S(z')$.
\end{itemize}

Define $K:B\times B\to \RR$ by $K(y,z)=1$ if $J(y)=z$ and $K(y,z)=0$ otherwise. Since $\nu \times c|_{\cR(B)} = c \times \nu|_{\cR(B)}$,
\begin{eqnarray*}
\nu(\tY) &=& \int \sum_{z\in \tZ} K(y,z) ~d\nu(y) = \int \sum_{y\in \tY} K(y,z) ~d\nu(z)\\
 &=& \int_{\tZ} |J^{-1}(z)| ~d\nu(z) \le C_d \nu(\tZ). \end{eqnarray*}
This implies the lemma.
\end{proof}


\begin{proof}[Proof of Theorem \ref{thm:maximaleq}]
We assume that $\cF$ is non-shrinking with constant $C_s>0$ and satisfies the doubling condition with constant $C_d>0$. For $n>0$, let 
$$\MM_n[\cF;f](b) := \max_{1\le i \le n} \bA_i[\cF;|f|](b).$$
Let $D_{n,t}:= \{b \in B:~\MM_n[\cF;f](b)>t\}$. It suffices to show that $\nu(D_{n,t}) \le \frac{C||f||_1}{t}$ for each $n>0$ where $C>0$ is a constant.

Let $\rho:D_{n,t} \to \NN$ be the function $\rho(b)=m$ if $m$ is the smallest integer such that $\bA_m[\cF;|f|](b)>t$. Let $Z \subset D_{n,t}$ be the subset given by the previous lemma where $Y=D_{n,t}$. Let $\tZ=\cup\{ \cF_{\rho(z)}(z):~z\in Z\}$ and $\tY=\cup\{ \cF_{\rho(y)}(y):~y\in Y\}$. The non-shrinking property of $\cF$ and the previous lemma imply 
$$C_s \nu(D_{n,t}) \le \nu (\tY) \le C_d\nu(\tZ).$$

The disjointness property of $Z$ implies that for every $z \in \tZ$ there exists a unique element $\pi(z) \in Z$ with $z \in \cF_{\rho(\pi(z))}(\pi(z))$. Since $Z \subset D_{n,t}$,
\begin{eqnarray*}
 \nu(D_{n,t}) \le C_s^{-1}C_d  \nu(\tZ) \le \frac{C_d}{C_s t} \int_{\tZ} \bA_{\rho(\pi(z))}[\cF;|f|](\pi(z)) ~dz.
 \end{eqnarray*}
Let $K:B \times B \to \RR$ be the function 
$$K(y,z)=\frac{|f(y)|}{|\cF_{\rho(\pi(z))}(\pi(z))|}$$
if $z \in \tZ$ and $y \in  \cF_{\rho(\pi(z))}(\pi(z))$. Let $K(y,z)=0$ otherwise. Since $\nu \times c|_{\cR(B)} = c \times \nu|_{\cR(B)}$,
\begin{eqnarray*}
\int_{\tZ} |f(y)| ~d\nu(y)&=& \int \sum_{z\in \tZ} K(y,z) ~d\nu(y) = \int \sum_{y\in \tZ} K(y,z) ~d\nu(z)\\
 &=& \int_{\tZ} \bA_{\rho(\pi(z))}[\cF;|f|](\pi(z)) ~d\nu(z).
\end{eqnarray*}
So
\begin{eqnarray*}
 \nu(D_{n,t}) \le \frac{C_d}{C_s t} \int_{\tZ} \bA_{\rho(\pi(z))}[\cF;|f|](\pi(z)) ~dz = \frac{C_d}{C_s t}\int_{\tZ} |f(y)| ~dy \le \frac{C_d||f||_1}{C_s t}.
    \end{eqnarray*}
This proves the theorem with $C= \frac{C_d}{C_s}$.

\end{proof}

\section{The free group and its boundary}\label{sec:free group review}

Let $\FF=\langle a_1,\dots ,a_r \rangle$ be the free group of rank $r\ge 2$. The {\em reduced form} of an element $g\in \FF$ is the expression  $g=s_1\cdots s_n$ with $s_i \in \sS$ and $s_{i+1}\ne s_i^{-1}$ for all $i$. It is unique. Define $|g|:=n$, the length of the reduced form of $g$. The distance function on $\FF$ is defined by $d(g_1,g_2):=|g_1^{-1}g_2|$.

\subsection{The boundary}

The boundary of $\FF$ can be represented in many equivalent forms. For example, it is the set of all geodesic rays in $\Gamma$ emanating from the origin. Alternatively, it can be described as the set of all sequences $\xi=(\xi_1,\xi_2,\ldots) \in \sS^\NN$ such that $\xi_{i+1} \ne \xi_i^{-1}$ for all $i\ge 1$. We denote it be $\partial \FF$. A metric $d_\partial$ on $\partial \FF$ is defined by $d_\partial\big( (\xi_1,\xi_2, \ldots), (t_1, t_2, \ldots) \big) = \frac{1}{n}$ where $n$ is the largest  natural number such that $\xi_i=t_i$ for all $i < n$. If $\{g_i\}_{i=1}^\infty$ is any sequence of elements in $\FF$ and $g_i:=t_{i,1}\cdots t_{i,n_i}$ is the reduced form of $g_i$ then $\lim_i g_i = (\xi_1,\xi_2,\ldots) \in \partial \FF$ if $t_{i,j}$ is eventually equal to $\xi_j$ for all $j$. If $\xi \in \partial \FF$ then we will write $\xi_i \in \sS$ for the elements in $\xi=(\xi_1,\xi_2,\xi_3,\ldots)$.

We define a probability measure $\nu$ on $\partial \FF$ as follows. For every finite sequence $t_1,\ldots, t_n$ with $t_{i+1} \ne t_i^{-1}$ for $1\le i <n$, let
$$\nu\Big(\big\{ (\xi_1,\xi_2,\ldots) \in \partial \FF :~ \xi_i=t_i ~\forall 1\le i \le n\big\}\Big) := |S_n(e)|^{-1}=(2r-1)^{-n+1}(2r)^{-1}.$$
By the Carath\'eodory extension Theorem, this uniquely extends to a Borel probability measure $\nu$ on $\partial \FF$.

\subsection{Horofunctions and horospheres}\label{sec:horospheres}
There is a natural action of $\FF$ on $\partial \FF$ by 
$$(t_1\cdots t_n)\xi := (t_1,\ldots,t_{n-k},\xi_{k+1},\xi_{k+2}, \ldots)$$ where $t_1,\ldots, t_n \in \sS$,  $t_1\cdots t_n$ is in reduced form and $k$ is the largest number $\le n$ such that $\xi_i^{-1} = t_{n+1-i}$ for all $i\le k$.  Observe that if $g=t_1 \cdots t_n$ then the Radon-Nikodym derivative satisfies
$$\frac{d\nu \circ g}{d\nu}(\xi) = (2r-1)^{2k-n}.$$
For $\xi \in \partial \FF$ as above, define the function $h_\xi:\FF \to \ZZ$ by 
$$h_\xi(g):=-\log_{2r-1}\left(\frac{d\nu \circ g^{-1}}{d\nu}(\xi) \right).$$
 For example, if $g=\xi_1\cdots \xi_n$ then $h_\xi(g)=-n$.  More generally, if $g=\xi_1\cdots \xi_n t_1 \cdots t_m$ is in reduced form and $t_1 \ne \xi_{n+1}$ then $h_\xi(g)=m-n$.  Alternatively, if $s_n=\xi_1\xi_2\cdots \xi_n$ then
 $$h_\xi(g) = \lim_{n \to \infty} d(g,s_n) - n.$$
The function $h_\xi$ is the {\em horofunction} associated to $\xi$. Figure \ref{fig:1} illustrates a horofunction.

A {\em horosphere} is any level set of a horofunction. Let $H_\xi$ denote the horosphere $H_\xi:=h_\xi^{-1}(0)$. Note 
$$H_\xi=\left\{g\in \partial \FF:~ \frac{d\nu \circ g^{-1}}{d\nu}(\xi) = 1\right\}.$$
If $\xi=(\xi_1,\xi_2,\ldots)$ then $g \in H_\xi$ if and only if the reduced form of $g$ is $$g=\xi_1\xi_2\cdots \xi_n t_1\cdots t_n$$ for some $t_1,\ldots, t_n \in \sS$ such that $\xi_{n+1}\ne t_1$ (so $g^{-1}\xi = (t_n^{-1},\ldots, t_1^{-1}, \xi_{n+1}, \ldots)$).  $H_\xi$ is called the {\em horosphere} passing through the identity $e$ associated to $\xi$. 

The group $\FF$ acts on horofunctions by $g\cdot h_\xi(f)=h_\xi(g^{-1}f)$ for any $g,f\in \FF, \xi \in \partial \FF$. The group also acts on horospheres by
$$g \cdot H_\xi = \{gf:~f\in H_\xi\}.$$
Observe that if $g \in H_\xi$ then $g^{-1}H_\xi = H_{g^{-1}\xi}$ and $g^{-1}\cdot h_\xi = h_{g^{-1}\xi}$. More generally, if $g \in \FF$ is arbitrary then
$$h_{g\xi} = g\cdot h_\xi - h_\xi(g^{-1}).$$

Let $\cR_0(\partial \FF)$ be the equivalence relation on $\partial \FF$ given by $(\xi,\eta) \in \cR_0(\partial \FF)$ if and only if there is a $g \in \FF$ such that $g\xi = \eta$ and $\frac{d\nu \circ g}{d\nu}(\xi) = 1$. In other words, $g^{-1}\in H_\xi$. By definition, $\nu$ is $\cR_0(\partial \FF)$-invariant.  



\begin{figure}
\begin{center}
 \includegraphics[width=1\textwidth]{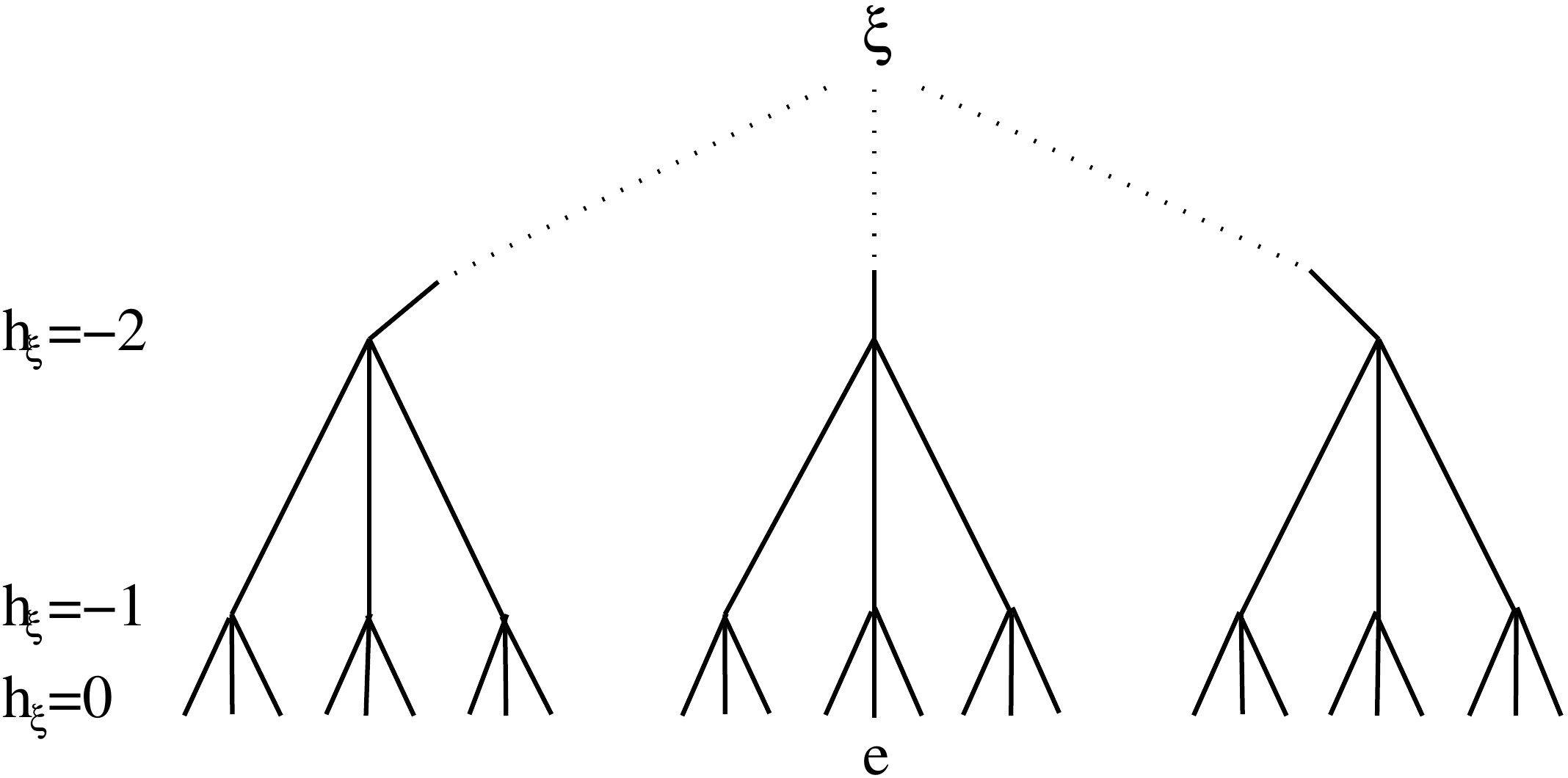}
\caption{The ``upper half space'' model of the rank 2 free group.}
\label{fig:1}
\end{center}
\end{figure}

%
%




\section{Horospherical sphere and ball averages}\label{sec:horo1}


 \subsection{Statement of main convergence result}

Suppose that $\FF$ acts on a standard probability space $(X,\lambda)$ by measure-preserving transformations. Then there is a natural equivalence relation $\cR_0(X\times \partial \FF)$ on $X \times \partial \FF$. Namely, $(x,\xi)$ is $\cR_0(X\times \partial \FF)$-equivalent to $(x',\xi')$ if there exists a $g\in H_\xi$ such that $g^{-1}x=x'$ and $g^{-1}\xi=\xi'$. Because $\nu$ is invariant under $\cR_0(\partial \FF)$, the product measure $\lambda \times \nu$ is $\cR_0(X\times \partial \FF)$-invariant. For a function $f\in L^1(X \times \partial \FF)$, let $\EE[f|\cR_0(X\times \partial \FF)]$ denote the conditional expectation of $f$ on the $\sigma$-algebra of $\cR_0(X\times \partial \FF)$-invariant sets. This is the {\em ergodic mean} of $f$.

For $n \ge 0$ and $(x,\xi) \in X \times \partial \FF$, let 
$$\tcS_n(x,\xi) := \{ (gx,g\xi) \in X \times \partial \FF :~ g^{-1}\in H_\xi, ~|g| = n\}.$$
This is the ``sphere of radius $n$ centered at $(x,\xi)$". Note that, if $n$ is odd and $\xi \in \partial \FF$ then there does not exist a $g^{-1} \in H_\xi$ with $|g|=n$. So $\tcS_n(x,\xi)$ is empty in this case.

The following is proven in the next section.
\begin{thm}[Pointwise convergence for horospherical sphere averages]\label{thm:hs}
For $n \ge 0$ let $\bA_{2n}[\tcS;\cdot]:L^1(X \times \partial \FF) \to L^1(X \times \partial \FF)$ be the operator given by
\begin{equation}
\bA_{2n}[\tcS;f](x,\xi) = \frac{1}{|\tcS_{2n}(x,\xi)|} \sum_{(x',\xi') \in \tcS_{2n}(x,\xi)} f(x',\xi').
\end{equation}
Then for any $f\in L^1(X \times \partial \FF)$, the sequence $\{\bA_{2n}[\tcS;f]\}_{n=1}^\infty$ converges pointwise a.e. and in $L^1$ norm to $\EE[f|\cR_0(X\times \partial \FF)]$. \end{thm}

By Theorem \ref{thm:pointwiseeq}, it suffices to prove that the sequence $\tcS:=\{\tcS_{2n}\}_{n=1}^\infty$ is asymptotically invariant, non-shrinking and doubling. This is accomplished in the next section.

\subsection{Balls and spheres}

\begin{defn}
For $g\in \FF$ and $n\ge 0$, let $B_n(g), S_n(g) \subset \F$  denote the ball and sphere respectively of radius $n$ centered at $g$ (with respect to the word metric). For $\xi \in \partial \FF$ and $x\in X$ let
\begin{eqnarray*}
\cB_n(\xi) &:=& \{g\xi:~ g^{-1}\in H_\xi \cap B_n(e)\},\\
\cS_n(\xi) &:=& \{g\xi:~ g^{-1}\in H_\xi \cap S_n(e)\},\\
\tcB_n(x,\xi) &:=& \big\{ (gx,g\xi) \in X \times \partial \FF ~|~ g^{-1}\in H_\xi \cap B_n(e)\big\},\\
\tcS_n(x,\xi) &:=& \big\{ (gx,g\xi)\in X \times \partial \FF~|~ g^{-1}\in H_\xi \cap S_n(e)\big\}.
\end{eqnarray*}
\end{defn}

Our goal is to prove that $\tcS$ is is asymptotically invariant, non-shrinking and doubling (Proposition \ref{prop:properties}). However, it is easier to first prove these properties for $\cB:=\{\cB_n\}_{n=1}^\infty$ and $\cS:=\{\cS_n\}_{n=1}^\infty$ and then transfer them to $\tcB$ and $\tcS$.
\begin{lem}\label{lem:basic}
Let $\xi=(\xi_1,\xi_2,\ldots) \in \partial \FF$. Then for any $n\ge 0$,
$$\cB_{2n}(\xi) = \{ (t_1,t_2, \ldots) \in \partial \FF:~ t_i=\xi_i ~\forall i >n\},$$
$$\cS_{2n}(\xi) = \{ (t_1,t_2,\ldots) \in \partial \FF:~ t_n \ne \xi_n \textrm{ and } t_i=\xi_i~ \forall i >n\}.$$
\end{lem}
\begin{proof}
This is an exercise.
\end{proof}

\begin{defn}
Let $\pi_n:\partial \FF \to S^n$ be the projection map $\pi_n((s_1,s_2,\ldots))=(s_1,\ldots,s_n)$. We will say that a map $\phi:\partial \FF \to \partial \FF$ has {\em order $n$} if it is measurable with respect to the $\sigma$-algebra generated by the inverse images of $\pi_n$. In other words, $\phi$ has order $n$ if for any two boundary points $\xi,\xi' \in \partial \FF$ such that $\pi_n(\xi)=\pi_n(\xi')$, $\phi(\xi)=\phi(\xi')$.  
\end{defn}

\begin{lem}
For any $(\xi,\xi') \in \cR_{0}(\partial \FF)$, there exists a map $\phi \in \Inn(\cR_0(\partial \FF))$ such that $\phi(\xi)=\xi'$ and $\phi$ has order $n$ for some $n<\infty$. Thus the set of finite order inner automorphisms of $\cR_0(\partial \FF)$ are generating.
\end{lem}

\begin{proof}
If $(\xi,\xi') \in \cR_0(\partial \FF)$ and $\xi=(\xi_1,\ldots)$, $\xi'=(\xi'_1,\ldots)$ then by definition of $\cR_0(\partial \FF)$, there is an $n$ such that $i\ge n$ implies $\xi_i = \xi'_i$. Let $m>n$ and let $\beta:S^m \to S^m$ be a bijection such that 
\begin{itemize}
\item $\beta(\pi_m(\xi))=\pi_m(\xi')$, 
\item $\beta$ does not change the last coordinate in the sense that if $\beta(s_1,\ldots,s_m)=(t_1,\ldots,t_m)$ then $s_m=t_m$,
\item $\beta$ maps the subset $\{ (s_1,\ldots,s_m):~ s_{i+1}\ne s_i^{-1} \forall 1\le i <m\}$ into itself.
\end{itemize}
Define $\phi \in \Inn(\cR_0(\partial \FF))$ by $\phi(s_1,s_2,\ldots) = (t_1,t_2,\ldots)$ where $\beta(s_1,\ldots,s_m)=(t_1,\ldots,t_m)$ and $s_i=t_i$ for $i\ge m+1$. It is easy to check that $\phi$ is an inner automorphism of finite order that maps $\xi$ to $\xi'$.
\end{proof}

\begin{lem}\label{lem:asymptoticinvariance}
The sequences $\cB:=\{\cB_{2n}\}_{n=1}^\infty$ and $\cS:=\{\cS_{2n}\}_{n=1}^\infty$ are asymptotically invariant.
\end{lem}

\begin{proof}
From Lemma \ref{lem:basic}, it follows that $\cB_{2n}$ is invariant under all inner automorphisms of order $\le n$ while $\cS_{2n}$ is invariant under all inner automorphisms of order $\le n-1$. So the previous lemma implies this lemma. 
\end{proof}

\begin{lem}\label{lem:doubling}
The sequences $\cB$ and $\cS$ are doubling with constants $1$ and $\frac{2r-1}{2r-2}$ respectively.
\end{lem}

\begin{proof}
It follows from Lemma \ref{lem:basic} that if $\xi, \xi'$ are such that $\cB_{2n}(\xi) \cap \cB_{2m}(\xi') \ne \emptyset$ and $n\ge m$ then $\cB_{2n}(\xi) \supset \cB_{2m}(\xi')$. Therefore
$$\bigcup \big\{\cB_{2m}(\xi'):~ m \le n,  \cB_{2m}(\xi') \cap \cB_{2n}(\xi) \ne \emptyset \big\} = \cB_{2n}(\xi).$$
This implies $\cB$ is doubling (with doubling constant $1$). 

It follows from Lemma \ref{lem:basic} that
\begin{eqnarray}\label{eqn:compare}
\frac{|\cS_{2n}(\xi)|}{|\cB_{2n}(\xi)|} = \frac{2r-2}{2r-1}.
\end{eqnarray}
Therefore,
\begin{eqnarray*}
&&\left|\bigcup \big\{\cS_{2m}(\xi'):~ m \le n,  \cS_{2m}(\xi') \cap \cS_{2n}(\xi) \ne \emptyset \big\}\right|\\
 &\le& \left|\bigcup \big\{\cB_{2m}(\xi'):~ m \le n,  \cB_{2m}(\xi') \cap \cB_{2n}(\xi) \ne \emptyset \big\}\right|\\
&=&|\cB_{2n}(\xi)| \le \frac{2r-1}{2r-2} |\cS_{2n}(\xi)|.
\end{eqnarray*}
This proves $\cS$ is doubling with doubling constant $\frac{2r-1}{2r-2}$. 
\end{proof}

\begin{lem}\label{lem:nonshrink}
The sequences $\cB$ and $\cS$ are non-shrinking.
\end{lem}
\begin{proof}
Let $Y \subset \partial \FF$ and  $\rho:Y \to \NN$ be a bounded Borel function. Let $Y^\cB = \bigcup\{ \cB_{\rho(y)}(y):~y\in Y\}$. Since $Y \subset Y^\cB$ we have $\nu(Y) \le \nu(Y^\cB)$. This proves $\cB$ is non-shrinking.


The proof that $\cS$ is non-shrinking is a little more involved. Let $Y^\cS = \bigcup\{ \cS_{\rho(y)}(y):~y\in Y\}$. By Lemma \ref{lem:covering}, there exists a Borel set $Z \subset Y$ such that
\begin{enumerate}
\item for all $z_1,z_2 \in Z$, $\cB_{\rho(z_1)}(z_1) \cap \cB_{\rho(z_2)}(z_2) = \emptyset$;
\item if $Z^\cB =  \bigcup_{z\in Z} \cB_{\rho(z)}(z) $, then 
$$\nu \Big(Z^\cB \Big) \ge  \nu \Big(Y^\cB \Big).$$
\end{enumerate}
Define $F:\cR_0(\partial \FF) \to \RR$ by $F(\xi,\xi')=0$ if $\xi' \notin Z^\cB$. Otherwise, there exists a unique $z \in Z$ such that $\xi' \in \cB_{\rho(z)}(z)$. In this case, define $F(\xi,\xi')=\frac{1}{|\cB_{\rho(z)}(z)|}$ if $\xi \in \cS_{\rho(z)}(z)$. Let $F(\xi,\xi')=0$ otherwise. Let
$$Z^\cS =  \bigcup_{z\in Z} \cS_{\rho(z)}(z).$$
Since $\nu \times c|_{\cR_0(\partial \FF)} = c \times \nu|_{\cR_0(\partial \FF)}$,
\begin{eqnarray*}
\nu(Z^\cS)&=&\int \Big(\sum_{\xi'} F(\xi,\xi') \Big) ~d\nu(\xi)\\
&=& \int \Big(\sum_{\xi} F(\xi,\xi') \Big) ~d\nu(\xi')  = \int_{Z^\cB} \frac{2r-2}{2r-1} ~d\nu(\xi') = \frac{2r-2}{2r-1}\nu(Z^\cB).
\end{eqnarray*}
The last equality uses (\ref{eqn:compare}). By the inclusions $Z^\cS \subset Y^\cS$ and $Y \subset Y^\cB$, we have
$$\nu(Y^\cS) \ge \nu(Z^\cS) = \frac{2r-2}{2r-1}\nu(Z^\cB) \ge \frac{2r-2}{2r-1} \nu(Y^\cB) \ge \frac{2r-2}{2r-1}\nu(Y).$$
This proves $\cS$ is non-shrinking.

\end{proof}

\begin{prop}\label{prop:properties}
The sequences $\tcB:=\{\tcB_{2n}\}_{n=1}^\infty$ and $\tcS:=\{\tcS_{2n}\}_{n=1}^\infty$ are asymptotically invariant, non-shrinking and doubling.
\end{prop}

\begin{proof}
The action of $\FF$ on $(\partial \FF,\nu)$ is essentially free: 
$$\nu\left( \{ \xi \in \partial \FF:~\exists g\in \FF\setminus\{e\} \textrm{ such that } g\xi=\xi\}\right) = 0.$$
For any inner automorphism $\phi \in \Inn(\cR_0(\partial \FF))$, there is an inner automorphism $\tphi \in \Inn(\cR_0(\partial \FF \times X))$ defined by $\tphi(x,\xi)=(gx,g\xi)$ where $g\in \FF$ is the unique element such that $\phi(\xi)=g\xi$. This is well-defined on a set of full measure.
 
 Let $\Phi \subset \Inn(\cR_0(\partial \FF))$ be a countable set that generates $\cR_0(\partial \FF)$ and let $\tPhi=\{\tphi:~\phi \in \Phi\}$. Then for $\nu$-a.e. $\xi \in \partial \FF$ if $(x,\xi)$ is $\cR_0(\partial \FF \times X)$-equivalent to $(x',\xi')$  there is a unique $g \in \FF$ with $(gx,g\xi)=(x',\xi')$ and a $\phi \in \langle \Phi \rangle$ with $\phi(\xi)=\xi'$. Then $\tphi(x,\xi)=(x',\xi')$ so $\tPhi$ is generating. 
 
The projection map $\partial \FF \times X \to \partial \FF$ maps $\tcB_{2n}(x,\xi)$ to $\cB_{2n}(\xi)$ and $\tphi(\tcB_{2n}(x,\xi))$ to  $\phi(\cB_{2n}(\xi))$ bijectively. So the asymptotic invariance of $\cB$, proven in Lemma \ref{lem:asymptoticinvariance}, implies the asymptotic invariance of $\tcB$. The proof that $\tcS$ is asymptotically invariant is similar.
 
For any $(x,\xi) \in \partial \FF \times X$, the projection map $\partial \FF \times X \to \partial \FF$ maps 
$$\cup \{\tcB_{2n}(x',\xi'):~\tcB_{2n}(x',\xi') \cap \tcB_{2n}(x,\xi)\ne \emptyset\} \mapsto \cup \{\cB_{2n}(\xi'):~\cB_{2n}(\xi') \cap \cB_{2n}(\xi)\ne \emptyset\}$$
bijectively. So the doubling property of $\cB$, proven in Lemma \ref{lem:doubling}, implies $\tcB$ is doubling. The proof that $\tcS$ is doubling is similar.
  
It is obvious that $\tcB$ is non-shrinking. The proof that $\tcS$ is non-shrinking is similar to the proof of the previous lemma.
 \end{proof}

\section{Automatic ergodicity}\label{sec:ergodicity}

Let $\FF$ act on a standard probability space $(X,\lambda)$ by measure-preserving transformations. Let $\FF^2$ be the subgroup of $\FF$ generated by words of length 2, which has index 2 in $\FF$. For any $f\in L^1(X)$, let $\EE[f|\FF^2] \in L^1(X)$ be the conditional expectation of $f$ on the $\sigma$-algebra of $\FF^2$-invariant measurable sets.

 
For $f\in L^1(X)$, define $i(f) \in L^1(X\times \partial \FF)$ by $i(f)(x,\xi)=f(x)$. The map $f \mapsto i(f)$ isometrically embeds in $L^1(X)$ into $L^1(X \times \partial \FF)$. 

The purpose of this section is to prove:
\begin{thm}\label{thm:expectation1}
 For any $f\in L^1(X)$, $\EE\big[i(f)|\cR_0(X\times \partial \FF)\big]= i(\EE[f|\FF^2])$.
\end{thm}
Similar results were proven in [Bo08] for all word hyperbolic groups.
 
We remark that it is necessary to consider the action of $\FF^2$ rather than $\FF$. For example, if $X$ is a two-point set, $\lambda$ is the uniform probability measure and all generators $\{a_1,\ldots,a_r\}$ of $\FF$ act nontrivially on $X$ then the action of $\FF$ on $X$ is ergodic but the equivalence relation $\cR_0(X\times \partial \FF)$ on $X \times \partial \FF$ is not.

In the next section we prove:
\begin{thm}\label{thm:expectation2}
For any $F\in L^1(X \times \partial \FF)$, $\EE[F|\FF^2] = \EE[F|\cR_0(X\times \partial \FF)]$ where $\EE[F|\FF^2]$ denotes the conditional expectation of $F$ on the $\sigma$-algebra of $\FF^2$-invariant sets where $\FF^2$ acts on $X \times \partial \FF$ diagonally: $g\cdot (x,\xi)=(gx,g\xi)~\forall g\in \FF, x\in X,\xi\in\partial \FF$.
\end{thm}
Theorem \ref{thm:expectation1} follows from the theorem above and the next lemma.
\begin{lem}\label{thm:expectation3}
For any $f\in L^1(X)$, $\EE[i(f)| \FF^2] = i\big(\EE[f|\FF^2]\big)$.
\end{lem}
\begin{proof}
Without loss of generality, we may assume that the action of $\FF^2$ on $(X,\lambda)$ is ergodic. It suffices to show that the diagonal action $\FF^2 \cc X \times \partial \FF$ is  ergodic. 

Let $\mu$ be the uniform measure on the generating set $\sS$. Then the action of $\FF$ on the Poisson boundary of the random walk determined by $\mu$ is canonically identified with $\FF \cc (\partial \FF,\nu)$ (e.g., see [Ka00]). Note that the support of the convolution $\mu*\mu$ generates $\FF^2$. Hence the action of $\FF^2$ on the Poisson boundary of the random walk determined by $\mu*\mu$ is identified with the action of $\FF^2$ on $(\partial \FF,\nu)$. By [AL05], this action is weakly mixing. This implies the diagonal action of $\FF^2$ on $(\partial \FF \times X, \nu \times \lambda)$ is ergodic.
\end{proof}



\subsection{Proof of automatic ergodicity}\label{sec:proof2}

We now turn to the proof of Theorem \ref{thm:expectation2}. We begin with a few definitions and a reduction.

\begin{defn}
Let $\cI \subset L^1(X\times \partial \FF)$ be the set of functions $f$ that are invariant under the relation, i.e., for all $\phi \in \Inn(\cR_0(X\times \partial \FF))$, $f\circ\phi=f$. 

Let $\cI_2 \subset L^1(X\times \partial \FF)$ be the set of functions $f$ such that for all $g\in \FF^2$ and a.e. $(x,\xi) \in X\times \partial \FF$, $f(x,\xi)=f(g (x,\xi))$. Theorem \ref{thm:expectation2} is equivalent to the statement $\cI=\cI_2$. 
\end{defn}

\begin{defn}

For $(x,\xi) \in X \times \partial \FF$, recall that $\xi=(\xi_1,\xi_2,\ldots)$. Define $\Par(x,\xi) \in X \times \partial \FF$ by $\Par(x,\xi)=\xi_1^{-1}(x,\xi)$. More generally, if $n\ge 1$ then let $\Par^n(x,\xi):=(\xi_1\cdots \xi_n)^{-1}(x,\xi)$.
\end{defn}

\begin{lem}
Let $f\in L^1(X\times \partial \FF)$. If $f \circ \Par^2 =f$ a.e. then $f\in \cI_2$. 
\end{lem}

\begin{proof}
Let $(x,\xi) \in X\times \partial \FF$ and $g=t_1\cdots t_{2n}\in \FF^2$ be in reduced form. By definition,
$$g\xi = (t_1,\ldots,t_{2n-k},\xi_{k+1},\xi_{k+2},\ldots)$$
where $k$ is the largest number such that $\xi_i^{-1}=t_{2n+1-i}$ for all $i\le k$. For any $x\in X$, if $k$ is even then $(gx,g\xi) \in \Par^{-(2n-k)}\Par^{k}(x,\xi)$. If $k$ is odd then $(gx,g\xi) \in \Par^{-(2n-k+1)}\Par^{k+1}(x,\xi)$. Thus if $f\circ \Par^2=f$ a.e. then $f\circ g = f$ a.e.. This implies the lemma.
\end{proof}

\begin{prop}\label{prop:reduction}	
To prove Theorem \ref{thm:expectation2}, it suffices to prove that $f\circ \Par^2=f$ for all $f\in \cI$.
\end{prop}
\begin{proof}
From the above it follows that if $f\circ\Par^2=f$ for all $f\in \cI$ then $\cI \subset \cI_2$. To see the reverse inclusion, let $(x,\xi), (x',\xi') \in X \times \partial \FF$ be $\cR_0(X\times \partial \FF)$-equivalent. By definition, there exists $g \in \FF$ such that $(x',\xi')=(gx,g\xi)$. As noted above, $g$ is necessarily in $\FF^2$. Thus if $f \in \cI_2$ then for a.e. pair $(x,\xi), (x',\xi')$ of $\cR_0(X\times \partial \FF)$-equivalent points of $X\times \partial \FF$, $f(x,\xi)=f(x',\xi')$, namely $f\in \cI$. This shows $\cI_2 \subset \cI$.
\end{proof}


The next proposition is the key geometric result in the proof of Theorem \ref{thm:expectation2}. Define $\Par_\partial:\partial \FF \to \partial \FF$ by $\Par_\partial(\xi)=\xi_1^{-1}\xi$. Recall that $d_\partial$ is a distance function on $\partial \FF$ defined by $d_\partial\big( (\xi_1,\xi_2, \ldots), (t_1, t_2, \ldots) \big) = \frac{1}{n}$ where $n$ is the largest  natural number such that $\xi_i=t_i$ for all $i < n$.

\begin{prop}\label{prop:ve}
Let $n>5$ be an integer. Then there exist measurable maps $\psi,\omega:\partial\FF \to \partial\FF$ such that
\begin{enumerate}
\item $\forall \xi \in \partial\FF$, $d_\partial\big(\psi  \omega(\xi), \Par_\partial^2 \omega(\xi)\big)=\frac{1}{n}$;
\item $\forall \xi \in\partial\FF$,  $d_\partial(\xi,\omega \xi )=\frac{1}{n}$;
\item the graphs of $\omega$ and $\psi$ are contained in $\cR_0(\partial\FF)$;
\item $\forall \xi \in \partial\FF$, $\exists g\in \FF$ such that $\psi \omega(\xi) = g \omega(\xi)$ and $\Par_\partial^2\omega(\xi) = g \xi$.
\item $\forall f \in L^1(\partial\FF)$, 
$$\max\left( \|f \circ \omega \|_1,  \|f \circ \psi \|_1\right) \le (2r-1)^2 \|f\|_1$$
where $r$ is the rank of the free group $\FF$.
\end{enumerate}
\end{prop}

\begin{proof}

Recall that $\cS=\{a_1,\ldots,a_r,a_1^{-1},\ldots,a_r^{-1}\}$ is the chosen generating set of $\FF$. Let $K:\cS^3\to \cS^3$ be a bijection so that for any $(s_{n-1},s_n,s_{n+1}) \in \cS^3$, $K(s_{n-1},s_n,s_{n+1})=(s_{n-1},s'_n,s_{n+1})$ for some $s'_n \notin \{s_{n-1}^{-1},s_n,s_{n+1}^{-1})$.

Define $\omega:\partial \FF \to \partial \FF$ by $\omega(s_1,s_2,\ldots)=(t_1,t_2,\ldots)$ where $t_i = s_i$ for all $i\ne n$ and $t_n=s'_n$ where  $K(s_{n-1},s_n,s_{n+1})=(s_{n-1},s'_n,s_{n+1})$. By its definition $\omega$ is invertible, Borel, $d_\partial(\xi,\omega(\xi)) = \frac{1}{n}$ for any $\xi \in \partial \FF$ and $\omega_*\nu=\nu$. Moreover since $\omega$ does not change the tail of the sequence (i.e., because $t_i=s_i$ for all sufficiently large $i$), the graph of $\omega$ is contained in $\cR_0(\partial\FF)$. Because $\omega$ is measure-preserving, $\|f \circ \omega\|_1 = \|f\|_1$ for any $f\in L^1(\partial\FF)$.

Define $\psi:\partial \FF \to \partial \FF$ by 
$$\psi\omega(s_1,s_2,\ldots) = (s_3,\ldots, s_{n-1}, s'_n, s_n^{-1}, s_n', s_{n+1},s_{n+2},\ldots)$$
where $K(s_{n-1},s_n,s_{n+1})=(s_{n-1},s'_n,s_{n+1})$. Because $\omega$ is invertible, $\psi$ is well-defined. 

Note that the $m$-th coordinate of $\psi\omega(s_1,s_2,\ldots)$ equals the $m$-th coordinate of $\omega(s_1,s_2,\ldots)$ if $m\ge n$. Therefore, the graph of $\psi$ is contained in $\partial\FF$. If $\xi =  (s_1,s_2,\ldots)$ then
$$\Par^2_\partial \omega \xi = (s_3,\ldots, s_{n-1}, s'_n, s_{n+1},\ldots).$$
Thus $d_\partial(\psi \omega \xi, \Par_\partial^2 \omega \xi) = \frac{1}{n}$. Note that $\Par^2_\partial\omega \xi = (s_3\cdots s_{n-1})s'_n(s_1\cdots s_n)^{-1}\xi$. Similarly, $\psi\omega \xi = (s_3\cdots s_{n-1})s'_n(s_1\cdots s_n)^{-1}\omega \xi$. This proves the fourth item.
 
We claim that $\psi$ is at most $(2r-1)^2$-to-1 (that is, for each $b \in \partial \FF$, $b$ has at most $(2r-1)^2$-preimages under $\psi$). Because $\omega$ is invertible, it suffices to show that $\psi\omega$ is at most $(2r-1)^2$-to-1. Suppose that $(u_1,u_2,\ldots) \in \partial \FF$ and 
$$\psi\omega(u_1,u_2,\ldots)= \psi\omega(s_1,s_2,\ldots) =  (s_3,\ldots, s_{n-1}, s'_n, s_n^{-1}, s_n', s_{n+1},s_{n+2},\ldots).$$
By definition of $\psi\omega$, $u_i=s_i$ for $i \ge 3$. Since there are $(2r-1)^2$ choices for $(u_1,u_2)$ the claim follows.

Since the graph of $\psi$ is contained in $\cR_0(\partial\FF)$ the claim implies $\|f\circ \psi\|_1 \le (2r-1)^2 \|f\|_1$ for all $f \in L^1(\partial\FF)$. 

\end{proof}

\begin{lem}
There exist measurable maps $\Phi_n,\Psi_n,\Omega_n:X \times \partial\FF \to X \times \partial\FF$ (for $n =6,7,\ldots$) such that
\begin{enumerate}
\item for all $f\in L^1(X\times \partial\FF)$, $\lim_{n\to\infty} \|f\circ \Psi_n \circ \Omega_n- f\circ \Par^2 \circ \Phi_n\|_1 =0$;
\item for all $f\in L^1(X\times \partial\FF)$, $\lim_{n\to\infty} \|f\circ \Omega_n-f\|_1=0$;
\item the graphs of $\Phi$ and $\Psi$ are contained in $\cR_0(X \times \partial\FF)$.
\end{enumerate}
\end{lem}

\begin{proof}
For $n>5$ an integer, let $\psi$ and $\omega$ be as in Proposition \ref{prop:ve}. Fix $(x,\xi) \in X\times \partial\FF$ and let $g_1,g_2 \in \FF$ be such that $g_1\xi = \omega(\xi)$ and $g_2\xi=\psi(\xi)$. Define $\Omega_n(x,\xi):=(x,g_1\xi)$, $\Phi_n(x,\xi):=(g_1x,g_1\xi)$ and $\Psi_n(x,\xi):=(g_2x,g_2\xi)$.

Since the graphs of $\psi$ and $\omega$ are contained in $\cR_0(\partial\FF)$, the graphs of $\Phi_n$ and $\Psi_n$ are contained in $\cR_0(X \times \partial\FF)$. Let $d_X$ be a metric on $X$ that induces its Borel structure and makes $X$ into a compact space. For $(x,\xi),(x',\xi') \in X\times \partial\FF$, define $d_*((x,\xi),(x',\xi')) = d_X(x,x')  + d_\partial (\xi,\xi')$. By the previous proposition, $d_*(\Omega_n(x,\xi), (x,\xi) ) = d_\partial(\omega(\xi),\xi) =1/n$. Also $d_*(\Psi_n \circ \Omega_n(x,\xi), \Par^2 \circ \Phi_n(x,\xi)) =1/n$.  So if $f$ is a continuous function on $X \times \partial\FF$ then the bounded convergence theorem implies
\begin{eqnarray*}
\lim_{n\to\infty} \|f\circ \Psi_n \circ \Omega_n- f\circ \Par^2 \circ \Phi_n\|_1 &=&0\\
\lim_{n\to\infty} \|f\circ \Omega_n-f\|_1&=&0.
\end{eqnarray*}
It follows from the previous proposition that the operators $f \mapsto f \circ \Omega_n$, $f\mapsto f \circ \Phi_n$ and $f\mapsto f \circ \Psi_n$ are all bounded for $f \in L^1(X\times \partial\FF)$ with bound independent of $n$. It easy to see that $f \mapsto f\circ \Par^2$ is also a bounded operator on $L^1(X\times \partial\FF)$. Since the continuous functions are dense in $L^1(X \times \partial\FF)$, this implies the lemma.
\end{proof}

We can now prove Theorem \ref{thm:expectation2}.

\begin{proof}[Proof of Theorem \ref{thm:expectation2}]
By Proposition \ref{prop:reduction}, it suffices to show that $f\circ \Par^2=f$ for every $f\in \cI$. Let $\Phi_n, \Psi_n, \Omega_n$ ($n=6,7,\ldots$) be as in the previous lemma. Because $f \in \cI$ and the graph of $\Psi_n$ is contained in $\cR_0(X\times \partial\FF)$, it follows that $f\circ \Psi_n=f$ for all $n$. An easy exercise shows that $\Par$ preserves the equivalence relation: if $(x,\xi)$ is $\cR_0(X\times \partial\FF)$-equivalent to $(y,\xi')$ then $\Par(x,\xi)$ is $\cR_0(X\times \partial\FF)$-equivalent to $\Par(y,\xi')$. It follows that $f\circ \Par^2 \in \cI$. So $f\circ\Par^2\circ \Phi_n = f\circ \Par^2$ for all $n$. We now have
\begin{eqnarray*}
\|f - f\circ\Par^2\|_1 &=& \|f - f\circ\Par^2 \circ \Phi_n\|_1\\
&\le& \|f - f \circ \Psi_n \circ \Omega_n\|_1 +\|f \circ \Psi_n \circ \Omega_n-f\circ\Par^2 \circ \Phi_n\|_1\\
& =& \| f-f\circ\Omega_n\|_1+\|f \circ \Psi_n \circ \Omega_n-f\circ\Par^2 \circ \Phi_n\|_1.
\end{eqnarray*}
The previous lemma now implies $f=f\circ \Par^2$ as claimed.

\end{proof}


\section{Proofs of ergodic theorems}\label{sec:group actions}

\subsection{Applying the convergence of spherical-horospherical averages}
Collecting results of the previous sections, we can now prove:

\begin{cor}\label{thm:horoaverages}
Let $\FF$ act by measure-preserving transformations on a probability space $(X,\lambda)$. For $f\in L^1(X)$, let $\EE[f|\FF^2]$ be the conditional expectation of $f$ with respect to the $\sigma$-algebra of $\FF^2$-invariant sets. Then for $\lambda \times \nu$-a.e. $(x,\xi) \in X\times \partial \FF$,
$$\EE[f|\FF^2](x) = \lim_{n\to\infty} \bA_{2n}[\tcS;i(f)](x,\xi)= \lim_{n\to\infty}  \frac{1}{|\tcS_{2n}(x,\xi)|} \sum_{(x',\xi') \in \tcS_{2n}(x,\xi)} f(x').$$
\end{cor}

\begin{proof}
By Theorem \ref{thm:hs}, for a.e. $(x,\xi) \in X\times \partial \FF$,
$$\EE[i(f)|\cR_0(X\times \partial \FF)](x,\xi) =  \lim_{n\to\infty}  \frac{1}{|\tcS_{2n}(x,\xi)|} \sum_{(x',\xi') \in \tcS_{2n}(x,\xi)} f(x').$$
By Theorem \ref{thm:expectation1}, $\EE\big[i(f)|\cR_0(X\times \partial \FF))\big] = i\big(\EE[f|\FF^2]\big)$. So, $\EE\big[i(f)|\cR_0(X\times \partial \FF))\big](x,\xi) = \EE[f|\FF^2](x)$ for a.e. $(x,\xi) \in X\times \partial \FF$.
\end{proof}

In the next section, we will need the following strong $L^p$-maximal inequality. Recall that for $f \in L^1(X \times \partial \FF)$,
$$\MM[\tcS;f]:=\sup_n \bA[\tcS;f].$$
\begin{prop}\label{prop:tmaximal}
For every $p>1$ there is a constant $C_p>0$ such that for every $f\in L^p(X\times \partial \FF)$, $\|\MM[\tcS;f]\|_p \le C_p \|f\|_p$. Moreover, there is a constant $C_1$ such that if $f\in L\log^+ L(X,\lambda)$, then 
$$\norm{\MM[\tcS;f]}_{L^1}\le C_1\norm{f}_{L\log L}.$$
 \end{prop}
 
\begin{proof}
It follows from Proposition \ref{prop:properties} and Theorem \ref{thm:maximaleq} that for any $f \in L^1(X \times \partial \FF)$ the weak-type (1,1) maximal inequality holds:
$$\lambda \times \nu\left(\left\{(x,\xi) \in X \times \partial \FF:~ \MM[\tcS;f] >t \right\}\right) \le \frac{C \|f \|_1}{t}$$
for some constant $C>0$.

The first part of the proposition now follows from standard interpolation arguments. Namely, 
since the operator $f\mapsto \MM[\tcS;f]$ is of weak-type $(1,1)$ and is norm-bounded on $L^\infty$, it is norm-bounded in every $L^p$, $1 < p < \infty$ (see e.g. \cite[Ch. V, Thm 2.4]{SW}).

Finally, given the weak-type $(1,1)$ maximal inequality, the fact that when $f\in L\log^+ L(X,\lambda)$, the maximal function is in fact integrable and satisfies the Orlicz-norm bound is standard, see e.g. \cite[p. 678]{DS}. 
\end{proof}

\subsection{Averaging over the boundary}

Throughout this section we let $1<p,q<\infty$ be such that $\frac{1}{p} + \frac{1}{q}=1$. Let $\psi \in L^q(\partial \FF,\nu)$ be a probability density on the boundary, namely $\psi\ge 0$ and $\int \psi ~d\nu =1$. For $f\in L^p(X,\lambda)$, recall that $i(f) \in L^p(X \times \partial \FF)$ is the function $i(f)(x,\xi)=f(x)$. 

The goal of this subsection is to prove:
\begin{prop}\label{prop:hat}
For $f\in L^p(X,\lambda)$ and $n \ge 0$, define $\bA_{2n}[\psi;f] \in L^p(X)$ by
$$\bA_{2n}[\psi;f](x):= \int_{\partial \FF}\bA_{2n}[\tcS;i(f)](x, \xi) \psi(\xi) ~d\nu(\xi).$$
Then $\bA_{2n}[\psi;f]$ converges pointwise a.e. to $\EE[f|\FF^2]$. Furthermore, if $\psi$ is essentially bounded then the same conclusion holds for any  $f\in L\log^+ L(X,\lambda)$. 
\end{prop}

The proof of Proposition \ref{prop:hat} uses the following :
\begin{lem}
Let $p,q,\psi, f$ be as above and define $\MM[\psi;f]:=\sup_n \bA_{2n}[\psi;f]$. Then there exists a constant $C_p>0$ (depending only on $p$) such that for every $f\in L^p(X,\lambda)$
$$\| \MM[\psi;f]\|_p \le C_p \|\psi\|_q \|f\|_p.$$
Furthermore, there is a constant $C_1>0$ such that if $\psi$ is bounded, then for any $f\in L\log^+ L(X,\lambda)$ we have 
$$\| \MM[\psi;f]\|_1 \le C_1 \|\psi\|_\infty \|f\|_{L\log L}.$$
\end{lem}

\begin{proof}
Let us start with the case $1 < p < \infty$. For a.e. $x\in X$,
\begin{eqnarray*}
\big| \MM[\psi;f](x) \big|^p &=&   \sup_n  \big|\bA_{2n}[\psi;f](x) \big|^p\\
&=& \sup_n \left| \int_{\partial \FF} \bA_{2n}[\tcS;i(f)](x,\xi) \psi(\xi) ~d\nu(\xi) \right|^p\\
&\le& \sup_n \|\bA_{2n}[\tcS;i(f)](x,\cdot)\|_{L^p(\partial \FF)}^p \|\psi\|_{L^q(\partial \FF)}^p.
\end{eqnarray*}
The last line above is justified by H\"older's inequality. Next, we observe that for any $n \ge 1$,
\begin{eqnarray*}
 \int_{X} \sup_n \|\bA_{2n}[\tcS;i(f)](x,\cdot)\|_{L^p(\partial \FF)}^p ~d\lambda(x)&=& \int_X\sup_n \int_{\partial \FF} |\bA_{2n}[\tcS;i(f)](x,\xi)|^p ~d\nu(\xi)d\lambda(x)\\
 &\le& \int_X\int_{\partial \FF} \MM[\tcS;i(f)](x,\xi)^p ~d\nu(\xi)d\lambda(x)\\
 &=& \| \MM[\tcS;i(f)] \|^p_{L^p(X\times \partial \FF)}.
 \end{eqnarray*}
 Putting this together with the previous inequality we obtain
 \begin{eqnarray*}
 \|\MM[\psi;f]\|_{L^p(X)}^p &=& \int_X \big| \MM[\psi;f](x) \big|^p  ~d\lambda(x)\\
 &\le& \| \MM[\tcS;i(f)] \|_{L^p(X\times \partial \FF)}^p\|\psi\|_{L^q(\partial \FF)}^p.
 \end{eqnarray*}
 The first part of the lemma now follows from 
 \begin{eqnarray*}
\|\MM[\psi;f]\|_{L^p(X)} &\le&  \|\psi\|_{L^q(\partial \FF)} \|\MM[\tcS;i(f)]\|_{L^p(X\times \partial \FF)}\\
 &\le& C_p \|\psi\|_{L^q(\partial \FF)} \|i(f)\|_{L^p(X\times \partial \FF)} = C_p \|\psi\|_{L^q(\partial \FF)} \|f\|_{L^p(X)}
  \end{eqnarray*} where $C_p>0$ is as in Proposition \ref{prop:tmaximal}.
  
  The second part of the lemma follows in exactly the same way, taking $f\in L\log^+ L(X)$, $p=1$ and $q=\infty$ above. Using the integrability of the maximal function and the norm bound  
  $$\norm{\MM[\tcS;i(f)}_{L^1(X\times \partial \FF)}\le C_1 \norm{f}_{(L\log L)(X)}\,$$
  together with the boundedness of $\psi$, the desired estimate follows. 
 \end{proof}

\begin{proof}[Proof of Proposition \ref{prop:hat}]
We will first prove the proposition in the special case in which $f\in L^\infty(X)$. By Corollary \ref{thm:horoaverages}, $\bA_{2n}[\tcS;i(f)]$ converges pointwise a.e. to $i\big(\EE[f|\FF^2]\big)$. By the Lebesgue dominated convergence theorem, this implies that for a.e. $x\in X$, $\{\bA_{2n}[\psi;f](x)\}_{n=1}^\infty$ converges  to $\int \EE[f|\FF^2](x)\psi(\xi) ~d\nu(\xi) = \EE[f|\FF^2](x)$. This finishes the case in which $f\in L^\infty(X)$.

Now suppose that $f\in L^p(X)$. After replacing $f$ with $f-\EE[f|\FF^2]$ if necessary, we may assume that $\EE[f|\FF^2]=0$. Let $\epsilon>0$. Let $f' \in L^\infty(X)$ be such that $\|f-f'\|_p <\epsilon$ and $\EE[f'|\FF^2]=0$. Clearly :
$$|\bA_{2n}[\psi;f]| \le |\bA_{2n}[\psi;f-f']| + |\bA_{2n}[\psi;f']|\le \MM[\psi;f-f'] + |\bA_{2n}[\psi;f']|.$$
Since $\bA_{2n}[\psi;f'] \to 0$ pointwise a.e., it follows that for a.e. $x\in X$,
$$\limsup_n |\bA_{2n}[\psi;f](x)| \le \MM[\psi;f-f'](x).$$
The previous lemma implies:
$$\| \limsup_n |\bA_{2n}[\psi;f]| \|_p \le \|\MM[\psi;f-f']\|_p \le C \|f-f'\|_p \le  C\epsilon.$$
Since $\epsilon>0$ is arbitrary, it follows that $\| \limsup_n |\bA_{2n}[\psi;f]| \|_p =0$. Equivalently, $\bA_{2n}[\psi;f]$ converges to $0$ pointwise a.e. 

The second part of the proposition follows similarly using approximation in the Orlicz norm. 

\end{proof}
We can now state the following corollary, proved previously for $L^p$, $p > 1$ in \cite{N1} \cite{NS} 
and for $L\log^+ L$ in \cite{Bu}. 
\begin{cor}\label{cor:thirdproof}
Let $p>1$ and $f\in L^p(X)$ or more generally $f\in L\log^+ L(X)$. Then for a.e. $x\in X$,
$$\EE[f|\FF^2](x) = \lim_{n\to\infty} \frac{1}{|S_{2n}(e)|} \sum_{g\in S_{2n}(e)} f(g^{-1}x).$$
\end{cor}

\begin{proof}
This follows from the previous proposition by setting $\psi \equiv 1$.
\end{proof}

\subsection{Comparing averages on the boundary}

We now turn to establish  that each operator $f\mapsto \bA_{2n}[\psi;f]$ has a form similar to that of the operator $\mu_{2n}$ from  
 Theorem \ref{thm:main}.  Namely, $f\mapsto \bA_{2n}[\psi;f]$ is given by averaging with respect to a probability measure on $\FF$. 
Recall that we have already associated with a probability density $\psi$ on the boundary a sequence of probability measures on the group, namely $\mu_{2n}^\psi$, from Theorem \ref{thm:sector}. The sequence of probability measures we define now is different and will be denoted $\eta_{2n}^\psi$. In order to define it we need some definitions.
\begin{defn}
Let $t_1\cdots t_{2n}=g$ be the reduced form of an element $g\in \FF^2$. Define
$$O'(g) =O(t_1\cdots t_n) - O(t_1 \cdots t_nt_{n+1}) \subset \partial \FF$$
where $O(\cdot)$ is as defined in the introduction. An elementary exercise reveals that $\xi \in O'(g)$ if and only if $h_\xi(g)=0$. 
 \end{defn}
Define $\eta^\psi_{2n} \in l^1(\F)$ by
$$\eta^\psi_{2n}(g) :=\frac{1}{(2r-2)(2r-1)^{n-1}} \int_{O'(g)} \psi~ d\nu$$
if $|g|=2n$ and $0$ otherwise.

\begin{lem}\label{A_2n^psi}
For any function $f \in L^p(X)$ ($p\ge 1$), any $n\ge 0$ and any $x\in X$,
$$\bA_{2n}[\psi;f](x)=\sum_{g\in \FF} f(g^{-1}x) \eta^\psi_{2n}(g).$$
\end{lem}

\begin{proof}
Let $\xi \in \partial \FF$. Recall that for $x\in X$,
$$\tcS_{2n}(x,\xi)=\big\{ (g^{-1}x,g^{-1}\xi)~:~ g \in H_\xi \cap S_{2n}(e) \big\}.$$
Lemma \ref{lem:basic} implies $|H_\xi \cap S_{2n}(e)| = (2r-2)(2r-1)^{n-1} $. So,
\begin{eqnarray*}
\bA_{2n}[\tcS;i(f)](x,\xi) &=& \frac{1}{|\tcS_{2n}(x,\xi)|} \sum_{(x',\xi') \in \tcS_{2n}(x,\xi)} i(f)(x',\xi')\\
&=&\frac{1}{(2r-2)(2r-1)^{n-1}}  \sum_{g\in H_\xi \cap S_{2n}(e)}  f(g^{-1}x).
\end{eqnarray*}
Thus,
\begin{eqnarray*}
\bA_{2n}[\psi;f](x) &=& \frac{1}{(2r-2)(2r-1)^{n-1}}  \int_{\xi \in \partial \FF} \sum_{g\in H_\xi \cap S_{2n}(e)} f(g^{-1}x) \psi(\xi) ~d\nu(\xi).
\end{eqnarray*}
Since $g \in H_\xi \cap S_{2n}(e)$ if and only if $\xi \in O'(g)$, it follows by switching the order of the summation and integral above that
 \begin{eqnarray*}
\bA_{2n}[\psi;f](x) &=& \frac{1}{(2r-2)(2r-1)^{n-1}} \sum_{g\in S_{2n}(e)} \int_{\xi \in O'(g)} f(g^{-1}x) \psi(\xi) ~d\nu(\xi) \\ 
&=&  \sum_{g\in S_{2n}(e)}\eta_{2n}^\psi (g)f(g^{-1}x)=\eta_{2n}^\psi(f).
\end{eqnarray*}
\end{proof}

Recall that $\pi_\partial :\ell^1(\F) \to L^1(\partial \FF,\nu)$ is the linear map satisfying $\pi_\partial(\delta_g)= \nu(O_g)^{-1}\chi_{O_g}$ and if $\psi \in L^1(\partial \FF,\nu)$ then $\mu_n^\psi \in \ell^1(\FF)$ is the function
$$\mu_n^\psi(g)=\int_{O(g)} \psi(\xi)~d\nu(\xi)$$
if $g$ is in the sphere $S_n(e)$. Otherwise $\mu_n^\psi(g)=0$. The set $O(g)$ is the shadow of $g$ (with light-source at $e$) defined in \S \ref{sec:state}. 


\begin{lem}\label{lem:convergence}
 Let $\psi\in L^q(\partial \FF,\nu)$.  Then $\{\pi_\partial(\eta^\psi_{2n})\}_{n=1}^\infty$ converges to $\psi$ in $L^q$-norm when  $1 \le q < \infty$, and uniformly if $\psi$ is continuous. 
\end{lem}

\begin{proof}
For $n\ge 1$, let  $\EE[\psi|\Sigma_n]$ be the conditional expectation of $\psi$ on $\Sigma_n$, the $\sigma$-algebra generated by $\{O(g):~g\in S_n(e)\}$. Thus 
$$\EE[\psi|\Sigma_n](\xi) = \frac{1}{\nu(O_g)} \int_{O(g)} \psi(\xi')~d\nu(\xi')$$
if $\xi \in O_g$ with $g\in S_n(e)$. Note that $\nu(O_g)=|S_n(e)|^{-1}=\frac{1}{(2r)(2r-1)^{n-1}}$, and 
that  $\EE[\psi|\Sigma_{n}](\xi)=\pi_\partial (\mu_n^\psi)(\xi)$.

By the martingale convergence theorem, $\EE[\psi|\Sigma_n]$ converges to $\psi$ in $L^q$ as $n\to \infty$. Noting that 
\begin{eqnarray*}
\pi_\partial(\eta^\psi_{2n}) &=& \frac{|S_{2n}(e)|}{(2r-2)(2r-1)^{n-1}} \Big(|S_n(e)|^{-1}\EE[\psi|\Sigma_n]-|S_{n+1}(e)|^{-1}\EE(\psi|\Sigma_{n+1})\Big) \\
&=& 
\frac{2r-1}{2r-2}\pi_\partial(\mu_n^\psi)-\frac{1}{2r-2}\pi_\partial(\mu_{n+1}^\psi)\,\,,
\end{eqnarray*}
convergence of $\pi_\partial(\eta_{2n}^\psi)$ to $\psi$ in $L^q$ follows immediately. When $\psi$ is continuous on the boundary it is uniformly continuous and then clearly 
$\EE[\psi|\Sigma_{n}](\xi)=\pi_\partial (\mu_n^\psi)(\xi)$ converges uniformly to $\psi$. 

\end{proof}

This next result is not needed for the main theorem; it seems interesting for its own sake.

\begin{prop}
As above, let $1<p,q<\infty$ be such that $\frac{1}{p} + \frac{1}{q}=1$. Let $f \in L^{p'}(X)$ for some $p'$ with $p<p'$. For $x\in X$ and $n\ge 0$, define $f_{x,2n} \in l^p(\FF)$ by $f_{x,2n}(g)=f(g^{-1}x)$ if $g\in S_{2n}(e)$ and $f_{x,2n}(g)=0$ otherwise. Let $\pi'_\partial(f_{x,2n}):\partial \FF \to \RR$ be the function 
$$\pi'_\partial (f_{x,2n}):= \sum_{g\in S_{2n}(e)} f_{x,2n}(g) \chi_{O(g)}.$$
Then, for a.e. $x\in X$, $\{\pi'_\partial(f_{x,2n})\}_{n=1}^\infty$ converges to the constant function $\xi \mapsto \EE[f|\FF^2](x)$ in the weak topology on $L^p(\partial \FF,\nu)$.
\end{prop}

\begin{proof}


For $\rho \in L^p(\partial \FF,\nu)$ and $\psi \in L^q(\partial \FF,\nu)$, let $\langle \rho, \psi \rangle :=  \int \rho \psi~d\nu$. It suffices to show that for any $\psi \in L^q(\partial \FF,\nu)$ and a.e. $x\in X$, $\langle \pi'_\partial (f_{x,2n}), \psi\rangle$ converges to $\EE[f|\FF^2](x) \int \psi ~d\nu$. By linearity, we may assume that $\psi \ge 0$ and $\int \psi ~d\nu =1$. Observe that
\begin{eqnarray}\label{eqn:thing3}
\langle \pi'_\partial (f_{x,2n}), \psi\rangle = \langle \pi'_\partial (f_{x,2n}), \pi_\partial (\eta^\psi_{2n} )\rangle + \langle \pi'_\partial (f_{x,2n}), \psi - \pi_\partial (\eta^\psi_{2n}) \rangle.
\end{eqnarray}
It follows from Proposition \ref{prop:hat}, that for a.e. $x\in X$, 
$$\langle \pi'_\partial (f_{x,2n}), \pi_\partial (\eta^\psi_{2n}) \rangle = \bA_{2n}[\psi;f]$$
converges to $\EE[f|\FF^2](x)$. It follows from the previous lemma that $\psi - \pi_\partial ( \eta^\psi_{2n})$ converges to zero in norm. Since $ \norm{\pi'_\partial(f_{x,2n})}_p$ involves the {\it uniform  spherical} average of $\abs{f}^p$, it follows from Corollary \ref{cor:thirdproof} that $\|\pi'_\partial(f_{x,2n})\|_p$ converges to $\EE\big[|f|^p|\FF^2\big](x)^{1/p}$ for a.e. $x\in X$, where we also use $p<p'$ to conclude that $|f|^p \in L^{p'/p}(X)$ with $p'/p >1$.

By H\"older's inequality,
$$\left|\left\langle \pi'_\partial (f_{x,2n}), \psi - \pi_\partial (\eta^\psi_{2n}) \right\rangle\right| \le \|\pi'_\partial (f_{x,2n})\|_p \|\psi - \pi_\partial (\eta^\psi_{2n})\|_q$$
tends to zero as $n\to\infty$. Thus equation (\ref{eqn:thing3}) implies the proposition.
\end{proof}

\begin{remark}
Typically, $\pi'_\partial(f_{x,2n})$ does not converge to $\EE[f|\FF^2](x)$ in norm. To see this, observe that $\|\pi'_\partial(f_{x,2n})\|_p$ converges to $\EE_2[|f|^p](x)^{1/p}$ (for a.e. $x\in X$). The norm of the constant function $\xi \mapsto \EE[f|\FF^2](x)$ is $|\EE[f|\FF^2](x)|$. Unless $f$ is constant on the ergodic component containing $x$, Jensen's inequality implies $\EE[|f|^p|\FF^2](x)^{1/p} \ne |\EE[f|\FF^2](x)|$. This uses $p>1$.
\end{remark}

\subsection{Proof of the main theorem}

We now turn to the proof of Theorem \ref{thm:main}, whose formulation we recall for the reader's convenience. 

{\bf Theorem 1.2.}{\it  \,\,\,Let $\{\mu_{2n}\}_{n=1}^\infty$ be a sequence of probability measures in $\ell^1(\FF)$ such that $\mu_{2n}$ is supported on $S_{2n}(e)$.
 Let $1< q < \infty$,  and suppose that $\{\pi_\partial(\mu_{2n})\}_{n=1}^\infty$ converges in $L^q(\partial \FF,\nu)$. Let $(X,\lambda)$ be a probability space on which $\FF$ acts by measure-preserving transformations.
If $f\in L^p(X)$, $1<p <\infty$ and $\frac{1}{p} + \frac{1}{q} < 1$, then the averages 
$$\mu_{2n}(f)(x) := \sum_{g \in \FF} f(g^{-1}x)\mu_{2n}(g)$$
 converge pointwise almost surely and in $L^p$-norm to $\EE[f|\FF^2]$.  Furthermore, if $q=\infty$ and  $\pi_\partial(\mu_{2n})$ converge uniformly, then pointwise convergence to the same limit holds for any $f$ in the Orlicz space $(L\log L)(X,\lambda)$. }

\begin{proof}[Proof of Theorem \ref{thm:main}]
To begin, we assume $1<q<\infty$. Let ${p'}>1$ be such that $\frac{1}{{p'}} + \frac{1}{q}=1$. Since $\frac{1}{p}+\frac{1}{q}<1$, it follows that ${p'}<p$. Let $f\in L^p(X)$. Choose a measurable version $\EE[f|\FF^2]$ of the conditional expectation. Let $\psi$ be the limit of  $\{\pi_\partial(\mu_{2n})\}_{n=1}^\infty$. Let $X' \subset X$ be the set of all $x\in X$ such that
\begin{eqnarray*}
\EE[f|\FF^2](x) &=& \lim_{n\to\infty} \frac{1}{|S_{2n}(e)|} \sum_{g\in S_{2n}(e)} f(g^{-1}x)\,\,\\
 &=& \lim_{n\to\infty} \bA_{2n}[\psi;f](x)\,\, , \\
\left(\EE([|f|^{p'}|\FF^2](x)\right)^{1/{p'}}&=&\lim_{n\to \infty}\left(\frac{1}{|S_{2n}(e)|} \sum_{g\in S_{2n}(e)}\abs{f(g^{-1}x)}^{p'}\right)^{1/p'}\,\,.
\end{eqnarray*}
 
By Proposition \ref{prop:hat} and Corollary \ref{cor:thirdproof}, $\lambda(X')=1$. For $x\in X'$ and $n> 0$, let $f_{x,2n} \in l^{p'}(\FF)$ be the function $f_{x,2n}(g) := f(g^{-1}x)$ if $g\in S_{2n}(e)$ and $f_{x,2n}(g):=0$ otherwise. By Lemma \ref{A_2n^psi} and H\"older's inequality for functions on $\FF$,
\begin{eqnarray*}
\big|\mu_{2n}(f)(x)-\bA_{2n}[\psi;f](x)\big| &=&\big| \sum_{g\in S_{2n}(e)} f(g^{-1}x) \big(\mu_{2n}(g)-\eta^\psi_{2n}(g)\big)\big|\\
&\le& \|f_{x,2n}\|_{\ell^{p'}(\FF)} \|\mu_{2n}-\eta^\psi_{2n}\|_{\ell^q(\FF)}.
\end{eqnarray*}
Recall that $\pi_\partial:l^1(\FF)\to L^1(\partial \FF,\nu)$ is defined by $\pi_\partial(\delta_g)=\nu(O_g)^{-1}\chi_{O(g)} = |S_{2n}(e)|\chi_{O(g)}$ if $\abs{g}=2n$. It now follows that:
\begin{eqnarray*}
\|\mu_{2n}-\eta^\psi_{2n}\|_{\ell^q(\FF)} &=&  \Big(\sum_{g\in S_{2n}(e)} |\mu_{2n}(g) - \eta^\psi_{2n}(g)|^q \Big)^{1/q}\\
&=&  \left(\sum_{g\in S_{2n}(e)} \abs{\frac{1}{\nu(O_g)}\int_{O_g}\left( \mu_{2n}(g)-\eta_{2n}^\psi(g) \right) d\nu}^q\right)^{1/q} \\
&\le&
 \Big(\sum_{g\in S_{2n}(e)} \nu(O_g)^{q-1} \int_{O_g} |\pi_\partial(\mu_{2n})(\xi) - \pi_\partial(\eta^\psi_{2n})(\xi)|^q  ~d\nu(\xi) \Big)^{1/q}\\
&=& |S_{2n}(e)|^{-1/p'} \|\pi_\partial(\mu_{2n})-\pi_\partial(\eta^\psi_{2n})\|_{L^q(\partial \FF,\nu)}.
\end{eqnarray*}
Combining this with the previous inequality, we have
\begin{eqnarray*}\label{eqn:finis}
\big|\mu_{2n}(f)(x)-\bA_{2n}[\psi;f](x)\big| &\le&  |S_{2n}(e)|^{-1/p'} \|f_{x,2n}\|_{l^{p'}(\FF)} \|\pi_\partial(\mu_{2n})-\pi_\partial(\eta^\psi_{2n})\|_{L^q(\partial \FF,\nu)}.
\end{eqnarray*}

The definition of $X'$ implies 
$|S_{2n}(e)|^{-1/p'}\|f_{x,2n}\|_{\ell^{p'}(\FF)}$ tends to $\EE[|f|^{p'}|\FF^2](x)^{1/{p'}}$ as $n\to\infty$. Lemma \ref{lem:convergence} implies that $\|\pi_\partial(\mu_{2n})-\pi_\partial(\eta^\psi_{2n})\|_{L^q(\partial\FF,\nu)}$ tends to zero as $n\to\infty$. So 
$$\lim_{n\to\infty} \big|\mu_{2n}(f)(x)-\bA_{2n}[\psi;f](x)\big| =0.$$
The definition of $X'$ now implies
$$\lim_{n\to\infty} \mu_{2n}(f)(x) = \EE[f|\FF^2](x).$$
 This proves the pointwise result if $1<q<\infty$. 
 
 As to the case $q=\infty$, uniform convergence of $\pi_\partial(\mu_{2n})$ implies that the limit function $\psi$ is continuous on the boundary. Therefore the second part of Lemma \ref{lem:convergence} gives the uniform convergence of $\pi_\partial(\eta_{2n}^\psi)$ to $\psi$, and thus also the convergence of $\|\pi_\partial(\mu_{2n})-\pi_\partial(\eta^\psi_{2n})\|_{L^\infty (\partial\FF,\nu)}$ to zero.  Corollary 
 \ref{cor:thirdproof} gives the convergence  of $|S_{2n}(e)|^{-1}\|f_{x,2n}\|_{\ell^{1}(\FF)}$  to 
 $\EE[|f| |\FF^2](x)$ if $f\in L \log L(X,\lambda)$.  Using these  two facts the same arguments used above establish the desired result also in the case when $p=p'=1$ and $q=\infty$ provided $f \in L\log L(X,\lambda)$.  

Finally, we note that the fact that $\mu_{2n}(f)$ converges to $\EE[f|\FF^2]$ in $L^p$-norm (if $p>1$) follows from the pointwise result by  a standard argument (e.g., see the end of the proof of Theorem \ref{thm:hs}).

\end{proof}

\end{document}